\documentclass[12pt]{article}
\usepackage{graphicx}
\usepackage{bm}

\usepackage[margin=0.9in]{geometry}

\usepackage{url,subfig,latexsym,amsmath,amssymb, amsfonts,enumerate,
  epsfig, graphics,times, amsthm,mathtools}
  \usepackage{hyperref}
\usepackage{bbm}

\usepackage{tikz,pgfplots,pgfplotstable}
\newtheorem{theorem}{Theorem}[section]
\newtheorem{definition}{Definition}[section]

 \newtheorem{remark}{Remark}[section]
  \newtheorem{lem}{Lemma}[section]
  \newtheorem{cor}{Corollary}[section]
  
\theoremstyle{definition}
\newtheorem{example}{Example}[section]

\newcommand{\mybox}[1]{%
  \setbox0=\hbox{#1}%
  \setlength{\@tempdima}{\dimexpr\wd0+13pt}%
  \begin{tcolorbox}[colframe=mycolor,boxrule=0.5pt,arc=4pt,
      left=6pt,right=6pt,top=6pt,bottom=6pt,boxsep=0pt,width=\@tempdima]
    #1
  \end{tcolorbox}
}

\newcommand{\dlim}{\displaystyle \lim\limits}
\newcommand{\dsum}{\displaystyle \sum\limits}

\renewcommand{\vec}[1]{\symbf{#1}}
\renewcommand\vec{\mathbf}
\newcommand{\transpose}[0]{{\mkern-2mu\top}}

\newcommand*\circled[1]{\textcircled{\small{#1}}}

\def\S{\mathcal S}
\def\C{\mathcal C}
\def\Re{\mathcal R}

\def\K{\mathcal K}
\def\Td1{T^{D,1}_{\{x_n\}}}
\def\Ts1{T^{S,1}_{\{x_n\}}}

\def\K{\mathcal{K}}

\def\Tid1{T^{D,1}_{\{x_n+\zeta_i\}}}
\def\Tis1{T^{S,1}_{\{x_n+\zeta_i\}}}
\def\Tjs1{T^{S,1}_{\{x_n+\zeta_j\}}}
\def\Tjd1{T^{D,1}_{\{x_n+\zeta_j\}}}

\newcount\colveccount
\newcommand*\colvec[1]{
        \global\colveccount#1
        \begin{pmatrix}
        \colvecnext
}
\def\colvecnext#1{
        #1
        \global\advance\colveccount-1
        \ifnum\colveccount>0
                \\
                \expandafter\colvecnext
        \else
                \end{pmatrix}
        \fi
}

\tikzset{every node/.style={auto}}
 \tikzset{every state/.style={rectangle, minimum size=0pt, draw=none, font=\normalsize}}
 \tikzset{bend angle=7}

\title{Slack Reactants: A State-Space Truncation Framework to Estimate Quantitative Behavior of the Chemical Master Equation}

\author{Jinsu Kim\thanks{Department of Mathematics, University of California, Irvine, jinsu.kim@uci.edu.} \quad Jason Dark\thanks{Department of Mathematics, University of California, Irvine} \quad German Eciso\thanks{Department of Mathematics, University of California, Irvine, Department of Developmental and Cell Biology, University of California, Irvine} \ \ and \ \ Suzanne Sindi\thanks{Department of Applied Mathematics, University of California, Merced}}

\begin{document}

\maketitle

%
%


\begin{abstract}

 State space truncation methods are widely used to approximate solutions of the chemical master equation. While most methods of this kind focus on truncating the state space directly, in this work we propose modifying the underlying chemical reaction network by introducing \emph{slack reactants} that indirectly truncate the state space. More specifically, slack reactants introduce an expanded chemical reaction network and impose a truncation scheme based on user defined properties, such as mass-conservation. This network structure allows us to prove inheritance of special properties of the original model, such as irreducibility and complex balancing.
 
We use the network structure imposed by slack reactants to prove the convergence of the stationary distribution and first arrival times. We then provide examples comparing our method with the stationary finite state projection and finite buffer methods. Our slack reactant system appears to be more robust than some competing methods with respect to calculating first arrival times. 

\end{abstract}

\maketitle


\section{Introduction}
Chemical reaction networks (CRN) are a fundamental tool in the modeling of biological systems, providing a concise representation of known chemical or biological dynamics. 
A CRN is defined by a family of chemical reactions of the form
\begin{equation*}
    \sum_{i=1}^n \alpha_{i} X_i \rightarrow \sum_{i=1}^n \beta_{i} X_i,
\end{equation*}
where $\alpha_{i}$ and  $\beta_{i}$ are the number of species $X_i$ consumed and produced in this reaction, respectively. The classical approach to modeling this CRN is to consider the concentrations $c(t)=(c_1(t),c_2(t),\dotsc,c_n(t))^\transpose$, where $c_i(t)$ is the concentration of species $X_i$ at time $t$, and to use a system of nonlinear differential equations to describe the evolution of the concentrations.  



Suppose we are interested in studying the typical enzyme-substrate system \cite{menten1913kinetik}, given by the CRN
\begin{equation}\label{eq:enzyme}
    S+E \rightleftharpoons D \rightarrow  P+E.
\end{equation}

Given an initial condition where the molecular numbers of both $D$ and $P$ are $0$, a natural question to ask is how long the system takes to produce the \emph{first} copy of $P$. Similarly, there exists a time in the future where $S$ and $D$ are fully depleted, resulting in a chemically-inactive system, and one can ask when this occurs. These are both quantities that the classical deterministic modeling 
leaves unanswered -- by considering only continuous concentrations, there is not a well-defined way to address modeling questions at the level of single molecules as the model assumes that all the reactions simultaneously occur within infinitesimally small time intervals.

Instead, by explicitly considering each individual copy of a molecule, we may formulate a continuous-time Markov chain. This stochastic modeling is especially important when the system consists of low copies of species, in which case the influence of intrinsic noise is magnified \cite{paulsson2004summing, enciso2019embracing,anderson2019discrepancies}.  Rather than deterministic concentrations $c(t)$, we consider the continuous-time Markov chain $X(t)=(X_1(t),X_2(t),\dotsc,X_n(t))^\transpose$ describing the molecular number of each species $X_i$ at time $t$.   


The questions regarding the enzyme-substrate system \eqref{eq:enzyme}, such as the time of the first production of $P$, simply correspond to the \emph{first passage times} of various combinations of states. For a continuous-time Markov process $X$, the first passage time to visit a set $K$ of system states is formally defined as $\tau=\inf \{ t\ge 0 : X(t)\in K\}$.
%
%
One can directly compute $E(\tau)$ by using the transition rate matrix $A$. With few exceptions (such as CRNs with only zero- or first-order reactions \cite{lopez2014exact}),  most chemical reaction networks of any notable complexity will have an intractably large or infinite state-space, i.e. they exhibit the curse of dimensionality. This direct approach can, therefore, suffer from the high dimensionality of the transition rate matrix. 

An alternative approach is to estimate the mean first passage time by generating sample trajectories with stochastic simulation algorithms such as the Gillespie algorithm \cite{gillespie2007stochastic}. This overcomes the curse of dimensionality since a single trajectory needs only to keep track of its current population numbers. 
Nevertheless, there still remain circumstances under which it is more efficient to numerically evaluate the exact solution of the chemical master equation -- in particular, when the Markov process rarely visits $K$ so that significantly long simulations may be required to sample enough trajectories to estimate the mean first passage time \cite{macnamara2008multiscale}. 

Fortunately, a broad class of state space reduction methods have recently been developed that allow for direct treatment of the transition rate matrix.  \textcolor{black}{These methods are based on the truncation-and-augmentation  of the state space \cite{seneta1967finite, tweedie1971truncation, munsky2006finite, gupta2017finite}, the finite buffer method \cite{cao2008optimal,cao2016accurate} and linear programming \cite{kuntz2018approximations,kuntz2019bounding}. A recent review summarizes truncation-based methods \cite{kuntz2019stationary}. The \emph{stationary finite state projection (sFSP)}\cite{gupta2017finite}, and the \emph{finite buffer method} \cite{cao2008optimal, cao2016accurate} are examples of such truncation-based methods, which we will describe in detail below. They all satisfy provable error estimates on the probability distributions when compared with the original distribution.}  On the other hand, each of these methods has potential limitations for estimating mean first passage times and other quantitative features.  For instance, using sFSP depends on the choice of a designated state, which can significantly alter the estimate for first passage times.  


In this paper we provide a new algorithm of state space reduction, the \emph{slack reactant method}, for stochastically modeled CRNs. In this algorithm we generate a new CRN from an existing CRN by adding one or multiple new species, so that the associated stochastic system satisfies mass conservation relations and is confined to finitely many states. In order to ensure equivalent dynamics to the original system, we define a mild form of non-mass action kinetics for the new system.  Since the state space reduction is implemented using a fully determined CRN, we can study the CRN using well-known tools of CRN theory such as deficiency zero \cite{}, Lyapunov functions \cite{}, etc, as long as they extend to this form of kinetics.  




In Section~\ref{subsec:algo} we provide an algorithm to produce a slack variable network given a desired set of mass conservation relations.  In addition to its theoretical uses, this algorithm allows to implement existing software packages such as CERENA \cite{kazeroonian2016cerena}, StochDynTools \cite{HespanhaDec06} and FEEDME \cite{FEEDME} for chemical reaction networks to generate quantities such as the moment dynamics of the associated stochastic processes, using the network structures as inputs. 


We employ classical truncation Markov chain approaches to prove convergence theorems for slack networks. For fixed time $t$, if a probability density of each slack system under conservation quantity $N$ converges to its stationary distribution uniformly in $N$, then the stationary distribution of the slack system converges to the original stationary distribution as $N$ tends to infinity. We further prove that under a uniform tail condition of first passage times, the mean first passage time of the original system can be approximated with slack systems confined on a sufficiently large truncated state space. Finally we show that the existence of Lyapunov function for an original system guarantees that all the sufficient conditions for the first passage time convergence are satisfied. \textcolor{black}{We also show that this truncation method is natural in the sense that a slack system admits the same stationary distribution up to a constant multiplication as the stationary distribution of the original system if the original system is complex balanced.}

This paper is outlined as follows. In Section~\ref{sec:slack reactant}, we introduce the slack reactant method and include several illustrative examples. In Section~\ref{sec:compare to}, we demonstrate that the slack method compares favorably with other state space truncation methods (sFSP, finite buffer method) when calculating mean first passage times.
We prove convergence in the mean first passage time, and other properties, in Section~\ref{sec:theorems} 
In Section~\ref{sec:practical examples}, we use slack reactants to estimate the mean first passage times for practical CRN models such as a Lotka-Volterra population model and a system of protein synthesis with a slow toggle switch.

\section{Stochastic chemical reaction networks}\label{subsec:CRN}
\textcolor{black}{
A chemical reaction network (CRN) is a graph that describes the evolution of a biochemical system governed by a number of \emph{species} $(\S)$ and \emph{reactions} $(\Re)$. Each node in the graph represents a possible state of the system and nodes are connected by directed edges when a single reaction transforms one state into another.
}
Each reaction consists of \emph{complexes} and a \emph{reaction rate constant}. For example, the reaction representing the transformation of  complex $\nu$ to complex $\nu'$ at rate $\kappa$ is written as follows:
\begin{equation}
    \nu \xrightarrow{\kappa} \nu'.
\end{equation}
 A complex, such as $\nu$, is defined as a number of each species $S_i \in \S$. That is, $\nu = (\nu_1,\nu_2,\dots,\nu_d)$ representing a complex $\sum_{i=1}^d\nu_i S_i$, where $\nu_i \geq 0$ are the \emph{stochiometric coefficients} indicating how many copies of each species $S_i \in \mathcal S$ belong in complex $\nu$. 
 The full CRN is thus defined by $(\S,\C,\Re,\Lambda)$ where $\C$ and $\Lambda$ represent the set of complexes and reaction propensities respectively.

When intrinsic noise plays a significant role for system dynamics, we use a continuous-time Markov chain $\vec X(t)=(X_1(t),X(2),\dots,X_d(t))$ to model the copy numbers of species $S_i$ of a reaction network. The stochastic system treats individual reactions as discrete transitions between integer-valued states of the system. The probability density for $X(t)$ is denoted as 
\begin{equation*}
    p_{\vec x_0}(\vec x,t) = P(X(t)=\vec x \ | \ X(0)=\vec x_0),
\end{equation*}
where $X(0)$ is the initial probability density. We occasionally denote by $p(\vec x)$ the probability omitting the initial state $\vec x_0$ when the contexts allows.
For each state $\vec x$, the probability density $p(\vec x)$ obeys the chemical master equation (CME), which gives the time-evolution of $\vec{p}(t)$ with a linear system of ordinary differential equations \cite{gillespie1992rigorous}:
\begin{equation}\label{eq:master}
    \frac{d}{dt}\vec{p}^{\transpose}(t)=\vec{p}^{\transpose}A.
\end{equation}
Here, the entry $A_{ij}$ ($i\neq j$) is the transition rate at which the $i$-th state transitions to $j$-th state. Letting $\vec x$ and $\vec x'$ be the $i$-th and $j$-th states, the transition rate from $\vec x$ to $\vec x'$ is 
\begin{align*}
    A_{ij}=\sum_{\substack{\nu \to \nu' \\ \vec x+\vec \nu' -\vec \nu  = \vec x'}} \lambda_{\nu \to \nu'}(\vec x), 
\end{align*}
where $\lambda_{\nu\to \nu'}$ is the \emph{reaction intensity} for a reaction $\nu \to \nu'$.
The diagonal elements of $A$ are defined as
\begin{equation*}
    A_{jj} = -\sum_{i \neq j} A_{ij}.
\end{equation*}

Normally, reaction intensities hold a standard property,
\begin{align}\label{eq:standard intensity}
    \lambda_{\nu\to \nu'}(\vec x) > 0 \quad \text{if and only if} \quad  x_i \ge \nu_i \quad \text{for each $i$}. 
\end{align}
A typical choice for $\lambda_{\to '}$ is \emph{mass-action}, which defines for $ \vec x=(x_1,x_2,\dots,x_d)$ and $\vec \nu=(\alpha_1,\alpha_2,\dots,\alpha_d)$,
\begin{align*}
    \lambda_{\nu \to \nu'}=\kappa_{\nu\to \nu'}\prod_{k=1}^d x_i^{(\alpha_i)}
\end{align*}
where $\kappa_{\nu\to \nu'}$ is the rate constant. Here we used the notation $m^{(n)}=m(m-1)(m-2)\cdots (m-n+1)\mathbbm{1}_{\{n \ge m\}},$ for positive integers $m$ and $n$.




\section{Construct of Slack Networks}\label{sec:slack reactant}
In this Section, we introduce the method of slack reactants, which adds new species to an existing CRN so that the state space of the associated stochastic process is truncated to a finite subset of states. This model reduction accurately approximates the original system as the size of truncation is large. We begin with a simple example to demonstrate the main idea of the method.
\subsection{Slack reactants for a simple birth-death model}
Consider a mass-action 1-dimensional birth-death model,
\begin{equation}\label{eq:bd process}
    \emptyset \xrightleftharpoons[\kappa_2]{\kappa_1} X.
\end{equation}

\noindent For the associated stochastic process $X$, the mass-action assumption defines reactions intensities as $\lambda_{\emptyset \to X}(x)=\kappa_1$ and $\lambda_{X\to \emptyset}(x)=\kappa_2 x$ for each reaction in \eqref{eq:bd process}. 
Note that the count of species $X$ could be any positive integer value as the birth reaction $\emptyset \to X$ can occur unlimited times.
Therefore the state space for this system consists of infinitely many states. 
Consider instead the CRN
\begin{equation}\label{eq:slack bd process}
    Y \xrightleftharpoons[\kappa_2]{\kappa_1} X
\end{equation}
where we have introduced the \emph{slack reactant} \(Y\). This new network admits \({X+Y}\) as a \emph{conservative law} since for each reaction either one species is degraded by one while the other is produced by one.

Since the purpose of this new network is to approximate the qualitative behavior of the original system \eqref{eq:bd process}, we minimize the contribution of the slack reactant $Y$ for modeling the associated stochastic system. Hence, we assign \(Y\) a special reaction intensity -- instead of $\lambda_{Y\to X}(x,y)=\kappa_1 y$ using mass-action, we choose
\begin{align}\label{eq:intesity of y}
\lambda_{Y\to X}(x,y)=\kappa_1 \mathbbm{1}_{\{y\ge 0\}},
\end{align}
 and we use the same intensity $\lambda_{X\to Y}(x,y)=\kappa_2 x$ for the reaction $X\to Y$. By forcing \(Y\) to have ``zero-order kinetics,'' we ensure that the computed rates remain the same throughout the chemical network except for on the imposed boundary. This choice of reaction intensities not only preserves the conservative law $X(t)+Y(t)=X(0)+Y(0)$ but also prevent the slack reactant $Y$ from having negative counts with the characteristic term $\mathbbm{1}_{\{y\ge 0\}}$.

\begin{figure*}[!htb]
\centering{
\includegraphics{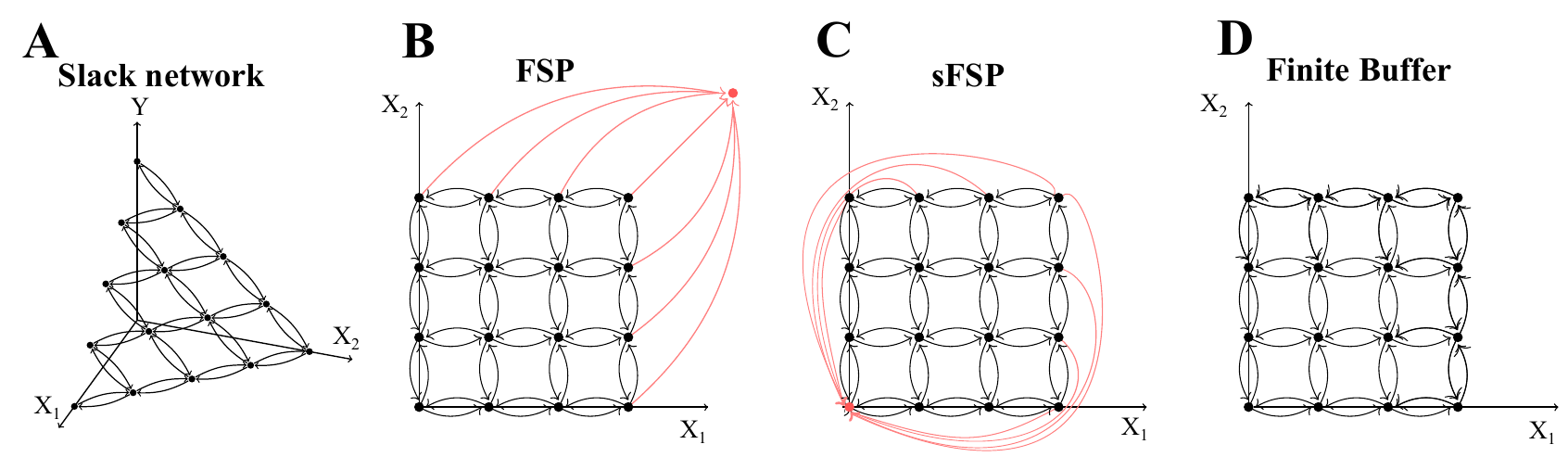}}
\caption{Schematic images of various state-space truncation methods for a CRN with two species: (a) Slack network (this paper), (c) Finite State Projection (FSP) \cite{munsky2006finite}, (c) Stationary Finite State Projection (sFSP) \cite{gupta2017finite}, (d) Finite Buffer Method \cite{cao2008optimal, cao2016accurate}.
}\label{fig:schematic}
\end{figure*}

\subsection{Algorithm for constructing a network with slack reactants}\label{subsec:algo}
In general, by using slack reactants, \emph{any} conservative laws can be deduced in a similar fashion to encode the desired truncation directly in the CRN. We have found this perspective to be advantageous with respect to studying complex CRNs. Rather than thinking about the software implementation of the truncation, it is often easier to design the truncation in terms of slack reactants, then implement the already-truncated system exactly.

We now provide an algorithm to automatically determine the slack reactants required for a specified truncation. It is often the case that a ``natural'' formulation arises (typically by replacing zero-order complexes with slack reactants) but when that is not the case, one can still find a useful approximation by following our algorithm.

Consider any CRN $(\S,\C,\Re,\Lambda)$. Suppose we wish to apply a number of \emph{conservation bounds} to reduce the state-space of the associated chemical master equation, e.g. many equations of the form
\begin{equation}
    \label{eq:ineq}
    w_{i,1} X_1 + w_{i,2} X_2 + \dotsc + w_{i,d} X_d \leq N_i,
\end{equation}
for $i=1,2,\dots,m$. Collecting the \emph{conservation vectors} $\{\vec w_i\}_{i=1}^m$ into the rows of the $m \times d$ matrix $W$ and each integer $\{N_i\}$ into the $m \times 1$ vector $\vec{N}$, we may concisely write
\begin{equation}\label{eq:L}
    W \vec{X} \preceq \vec{N},
\end{equation}
where $\preceq$ denotes an element-wise less-than-or-equal operator and $\vec{X}$ is the column vector of reactants ${\{X_i\}}$.

We augment the $d\times 1$ vector $\vec{X}$ with a $m \times 1$ vector $\vec{Y}$ of slack reactants to convert the inequality to an equality:
\begin{equation}\label{eq:conv}
    \begin{pmatrix} W & I \end{pmatrix} \begin{pmatrix} \vec{X} \\ \vec{Y} \end{pmatrix} = \vec{N},
\end{equation}
where $I$ is the $m\times m$ identity matrix. This matrix multiplication implies $m$ conservative laws among $X_i$'s and slack reactants $Y_i$'s. In other words, we inject the variables $\vec{Y}=(Y_1,Y_2,\dots,Y_m)$ into the existing reactions in such a way as to enforce these augmented conservation laws.

To do so, we need to properly define several CRN quantities. 
Let $S$ be the $|\C| \times |\Re|$ \emph{connectivity matrix}: if the $r$-th reaction establishes an edge from the $i$-th complex to the $j$-th, then $-S_{ir}=S_{jr}=1$ and the other entries in column $r$ are 0. Let $C$ be the $|\S|\times |\C|$ \emph{complex matrix} $C$. Each column of $C$ corresponds to a unique complex and each row corresponds to a reactant. The entry $C_{ij}$ is the multiplier of $X_i$ in complex $j$ (which may be implicitly 0). 
The \emph{stoichiometry matrix} is simply ${\Gamma=CS}$, whose $i$-th column indicates the reaction vector of the $i$-th reaction. Note that the CRN admits a conservation law $r_1X_1(t)+r_2X_2(t)+\cdots + r_dX_d(t)=r_1X_1(0)+r_2X_2(0)+\cdots + r_d X_d(0)$ if and only if $\vec r^\transpose \Gamma=\vec 0$, where $r=(r_1,r_2,\dots,r_d)^\transpose$.

Now consider a modified network where we inject $Y_1,Y_2,\dots, Y_K$ into the complexes. Let $D_{ij}$ denote the (unknown) multiplier of $Y_i$ in the $j$-th complex.  We require the following to be true:
\begin{equation}\label{eq:new CRN}
    \begin{pmatrix} W & I \end{pmatrix} \begin{pmatrix} C \\ D \end{pmatrix} S = 0.
\end{equation}
Then by using the same connectivity matrix $S$ and a new complex matrix $\widetilde C= \begin{pmatrix} C \\ D \end{pmatrix}$, we obtain a new CRN that enforces the conservative laws \eqref{eq:conv}.
To find $D$ satisfying \eqref{eq:new CRN}, we choose $D$ such that
\begin{equation}\label{eq:D}
    D = U-WC
\end{equation}
for an $m \times |\C|$ matrix $U$ with each row of the form $u_i \mathbf{1}^\top=u_i(1,1,\dots,1)^\transpose$ for some positive integer $u_i$. Since $\vec{1}^T S=\vec 0$, this guarantees \eqref{eq:new CRN}. 

A new network can be generated with the complex matrix $\widetilde C$ and the connectivity matrix $S$. We call this network a \emph{regular slack network} with species $\widetilde \S$ and complexes $\widetilde \C$. For the $j$-th reaction $\nu_j \to \nu'_j \in \Re$ corresponding to the $j$-th column of the matrix $CS$, we denote by $\tilde \nu_j \to \tilde \nu'_j \in \widetilde \Re$ the reaction corresponding to the $j$-th column of the matrix $\widetilde C S$. That is, $\tilde \nu_j\to \tilde \nu'_j$ is the reaction obtained from $\nu_j \to \nu'_j$ by adding slack reactants. 

We model the slack network by $X^{\vec N}(t)=(X^{\vec N}_1(t),  X^{\vec N}_2(t),\dots, X^{\vec N}_d(t))$ where each of the entries represents the count of species in the new network. Note that the count of each slack reactant $Y_i$ is fully determined by species counts $X^{\vec{N}}_i$'s because of the conservation law(s). As such, we do not explicitly model the $Y_i$'s.

 The intensity function $\lambda^{\vec N}_{\tilde \nu \to \tilde\nu}$ of $X^{\vec N}$ for a reaction $\tilde \nu \to \tilde \nu'$ is defined as 
\begin{align}\label{eq:slack intensity}
   \lambda^{\vec N}_{\tilde \nu \to \tilde \nu'}(\vec x)=\lambda_{\nu\to \nu'}(\vec x) \prod_{i=1}^m \mathbbm{1}_{\{y_i \ge \tilde \nu_{d+1}\}},
\end{align}
where $y_i=N_i-(w_{i1}x_1+w_{i1}x_1+\cdots+w_{id}x_d)$, and $\nu\to \nu'$ is the reaction in $\Re$.
Then we denote by $(\widetilde \S, \widetilde \C, \widetilde \Re,\Lambda^{\vec N})$ a new system with slack reactants obtained from the algorithm, where $\K_N$ is the collection of kinetic rates $\{ \lambda^N_{\tilde \nu \to \tilde \nu'} : \tilde \nu \to \tilde \nu' \in \widetilde \Re\}$. We refer this system with slack reactants to a \emph{slack system}.

Here we highlight that the connectivity matrix $S$ and the complex matrix $C$ of the original network are preserved for a new network with slack reactants. Thus the original network and the new network obtained from the algorithm have the same connectivity. This identical connectivity allows the qualitative behavior of the original network, which solely depends on $S$ and $C$, to be inherited to the new network with slack reactants. We particularly exploit the inheritance of accessibility and the inheritance of Poissonian stationary distribution in Section \ref{sec:acc}.

\begin{remark}
\textcolor{black}{A single conservation relation, such as $\vec w\cdot X +y= N$ with a non-negative vector $\vec w$, is sufficient to truncate the given state space into a finite state space. Hence, in this manuscript, we mainly consider a slack network that is obtained with a single conservation vector $\vec w$ and a single slack reactant $Y$. }
\end{remark}

\begin{remark}
Although we primarily think about bounding our species counts from above, we could also bound species counts from below by choosing negative integers for $w_{i,j}$.
\end{remark}

\begin{example}\label{ex:why intrusive}
We illustrate our slack algorithm with an example CRN consisting of two species and five reactions indicated by edge labels on the following network:
\begin{equation}\label{eq:ex original}
   \begin{tikzpicture}[baseline={(current bounding box.center)}, scale=0.8,  state/.style={circle,inner sep=2pt}]
   \node[state] (1) at (0,3)  {$\emptyset$};
   \node[state] (2) at (4,3)  {$A$};
   \node[state] (3) at (0,-1)  {$B$};

   \path[->]
    (1) edge[bend left] node   { \circled{1}} (2)
    (2) edge[bend left] node   {\circled{2} } (1)
    (2) edge[bend left] node  { \circled{3}} (3)
    (3) edge[bend left] node  { \circled{4}} (2)
    (1) edge node  { \circled{5}}  (3)
    ;
     \end{tikzpicture}
\end{equation}
We enumerate the complexes in the order of $\emptyset,A$ and $B$. We order the reactions according to their labels on the network. Thus the connectivity matrix $S$ and complex matrix $C$ are defined as follows:
$$S=\begin{bmatrix}
-1&1&0&0&-1\\
1&-1&-1&1&0\\
0&0&1&-1&1\\
\end{bmatrix},$$
$$C=\begin{bmatrix}
0 &1& 0\\
0 & 0 &1
\end{bmatrix}.$$  Suppose we set a conservation bound $A+B\le N$ for some $N>0$. Then the matrix $A=\begin{bmatrix}
1&1
\end{bmatrix}$ and by \eqref{eq:D}, we have $D=\begin{bmatrix}
u&u&u
\end{bmatrix}-\begin{bmatrix}
0&1&1
\end{bmatrix}$. 

\noindent When $u=1$, the network with slack reactant $Y$ is
\begin{equation}\label{eq:ex slack1}
   \begin{tikzpicture}[baseline={(current bounding box.center)}, scale=0.8,  state/.style={circle,inner sep=2pt}]
   \node[state] (1) at (0,2)  {$Y$};
   \node[state] (2) at (2,2)  {$A$};
   \node[state] (3) at (0,0)  {$B$};

   \path[->]
    (1) edge[bend left] node   {} (2)
    (2) edge[bend left] node   { } (1)
    (2) edge[bend left] node  { } (3)
    (3) edge[bend left] node  { } (2)
    (1) edge node  { }  (3)
    ;
     \end{tikzpicture},
\end{equation}
where we have the conservation relation $A+B+Y=N$, where $N=A(0)+B(0)+Y(0)$. 

When $u=2$, the network with slack reactant $Y$ is
\begin{equation}\label{eq:ex slack2}
   \begin{tikzpicture}[baseline={(current bounding box.center)}, scale=0.8,  state/.style={circle,inner sep=2pt}]
 \node[state] (1) at (0,2)  {$2Y$};
   \node[state] (2) at (3,2)  {$A+Y$};
   \node[state] (3) at (0,0)  {$B+Y$};

   \path[->]
    (1) edge[bend left] node   {} (2)
    (2) edge[bend left] node   { } (1)
    (2) edge[bend left] node  { } (3)
    (3) edge[bend left] node  { } (2)
    (1) edge node  { }  (3)
    ;
     \end{tikzpicture},
\end{equation}
where we have the same conservation relation $A+B+Y=N$.

Here we explain why network \eqref{eq:ex slack1} is preferred to \eqref{eq:ex slack2} because it is less intrusive.
Let $X=(X_A,X_B,X_Y)$ and $X'=(X'_A,X'_B,X'_Y)$ be the stochastic process associated with networks \eqref{eq:ex slack1} and \eqref{eq:ex slack2} respectively. Suppose the initial state of both $X$ and $X'$ is $(0,0,N)$. $X$ can reach the state $(0,N,0)$ by transitioning $N$ times with the reaction $Y\to B$. This state corresponds to the state $(0,N)$ in the original network \eqref{eq:ex original}. On the other hand, $X'$ cannot reach the state $(0,N,0)$. This is because the states $(0,0,N)$ and $(0,N-1,1)$ are the only states from which $X'$ jumps to $(0,0,N)$. However, no reaction in \eqref{eq:ex slack2} can be fired at the states since no species $Y$ presents at those states.

Consequently, one state, that is accessible in the original network \eqref{eq:ex original}, is lost in the system associated with network \eqref{eq:ex slack2}. However, it can be shown that the stochastic process associated with network \eqref{eq:ex slack1} preserves all the states of the original network. This occurs mainly because the matrix $D$ for network \eqref{eq:ex slack1} is sparser than the matrix $D$ for network \eqref{eq:ex slack2}.
We discuss how to minimize the effect of slack reactants in Section \ref{subsec:properties of slack}.
\end{example}

\subsection{Optimal Slack CRNs for effective approximation of the original systems}\label{subsec:properties of slack}


The algorithm we introduced to construct a network with slack reactants provides is valid and unique up to any user-defined conservation bounds \eqref{eq:ineq}, and the outcome is the matrix $D$, shown in \eqref{eq:D}, that indicates the stoichiometric coefficient of slack reactants at each complex.

As we showed in Example \ref{ex:why intrusive}, to minimize the `intrusiveness' of a slack network, we can simplify a slack network by setting as many ${D_{ij}=0}$ as possible. To do that, we choose the entries of $\vec{u}$ such that $u_i$ is the maximum entry of the $i$-th row of $AC$. We further optimize the effect of the slack reactants by removing the `redundant' stoichiometric coefficient of slack reactants. For example, for a CRN,
\begin{align}\label{eq:original}
    \emptyset \rightleftharpoons A \to 2A,
\end{align}
the algorithm generates the following new CRN with a single slack reactant $Y$ 
\begin{align}\label{eq:non breaking}
    2Y \rightleftharpoons A+Y \to 2A.
\end{align}

However, by breaking up the connectivity, we can also generate another network
\begin{align}\label{eq:breaking}
    Y \rightleftharpoons A, \quad A+Y \to 2A.
\end{align}
The network in \eqref{eq:non breaking} is more intrusive than the network in \eqref{eq:breaking} in the sense of accessibility. At any state where $Y=0$, the system associated with \eqref{eq:non breaking} will remain unchanged  because no reaction can take place. However, the reaction $A\to Y$ in \eqref{eq:breaking} can always occur despite $Y=0$. Hence \eqref{eq:breaking} preserves the accessibility of the original system associated with \eqref{eq:original} as any state for $A$ is accessible from any other state in the original reaction system \eqref{eq:original}. We refer such a system with slack reactants generated by canceling redundant slack reactants to an \emph{optimized slack system}.
In Section \ref{sec:acc}, we explore the accessibility of an optimized reaction network with slack reactants in a more general setting.

Finally, we can make a network with slack reactants admit a better approximation of a given CRN by choosing an optimized conservation relation in \eqref{eq:ineq}. First, we assume that only a single conservation law and a single slack reactant are added to a given CRN. For the purpose of state space truncation onto finitely many states, a single conservation law is enough as all species could be bounded by $N$ as shown in \eqref{eq:ineq}. Let this single conservation law be \[w_1 X_1+w_2 X_2 + \cdots  w_d X_d+Y=N.\] Then the matrix $W$ in \eqref{eq:L} is a vector $(w_1,w_2,\dots,w_d)^\transpose$, and we denote this by $\vec w$. By the definition of the intensities \eqref{eq:slack intensity} for a network with slack reactants, some reactions are turned off when $Y =0$, i.e. $w_1X_1+\dots+w_d X_d=N$. Geometrically, a reaction outgoing from the hyperplane  $w_1X_1+\dots+w_d X_d=N$ is turned off (Figure \ref{fig:example}).

Hence we optimize the estimation with slack reactants by minimizing such intrusiveness of turned off reactions. To do that, we choose $\vec v$ which minimizes the number of the reactions in  $\{ \nu\to \nu' \in \Re : (\nu'-\nu)\cdot \vec w > 0\}$.

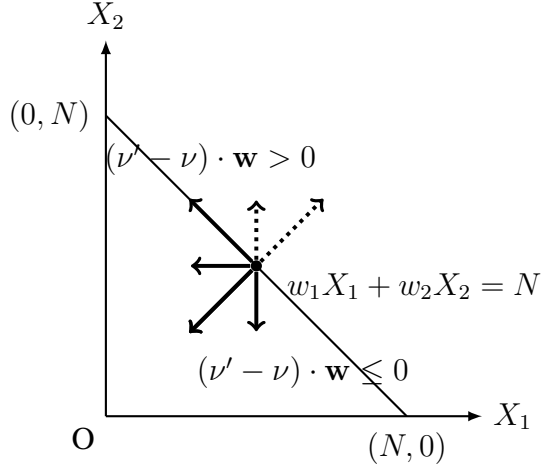
\begin{figure}
\centering
\begin{tikzpicture}
        \draw[thick,-latex] (0,0) -- (5,0)node[right]{$X_1$};
        \draw[thick,-latex] (0,0) -- (0,5)node[above]{$X_2$};
        \node at (-0.3,-0.3) {O};
    
            \node[inner sep=1.5pt,label=below:{$(N,0)$}] (-11) at (4,0) {};
            \node[inner sep=1.5pt,label=left:{$(0,N)$}] (0)at (0,4) {};
            \node[fill,circle,inner sep=1.5pt,label=below left:{}] (1) at (2,2) {};
             \node[inner sep=1.5pt,label=below left:{}] (2) at (1,2) {};
             \node[inner sep=1.5pt,label=below left:{}] (3) at (1,1) {};
             \node[inner sep=1.5pt,label=below left:{}] (4) at (2,1) {};
                \node[inner sep=1.5pt,label=below right:{$(\nu'-\nu) \cdot \vec w\le 0$}] (5) at (1,1) {};
                \node[inner sep=1.5pt,label=above left:{\hspace{1cm}\textcolor{black}{$ (\nu'-\nu)\cdot \vec{w}> 0$}}] (6) at (3,3) {};
                 \node[inner sep=1.5pt,label=below left:{}] (7) at (2,3) {};
                   \node[inner sep=1.5pt,label=below left:{}] (8) at (1,3) {};
                   
        \draw[thick,-, shorten >= 0cm, shorten <= 0cm] (0,4) -- (4,0) 
        node[pos=0.65,above] {$\hspace{3cm} w_1X_1+w_2X_2 = N$};
       
        \path[shorten >=2pt,->,shorten <=0pt]
        (1) edge[line width=1.5pt] node {} (2)
         (1) edge[line width=1.5pt] node {} (3)
          (1) edge[line width=1.5pt] node {} (4)
           (1) edge[line width=1.5pt] node {} (5)
               (1) edge[line width=1.5pt] node {} (8)
            (1) edge[dotted,line width=1.5pt] node {} (6)
            (1) edge[dotted, line width=1.5pt] node {} (7);

        \end{tikzpicture}

        \caption{The dotted arrows correspond to the reaction vectors that are turned off (i.e. the associated reaction intensity is zero) at the bound of the state space for a slack system as described in Section~\ref{subsec:properties of slack}.}\label{fig:example}
       \end{figure}

\section{Comparison to Other Truncation-Based Methods}\label{sec:compare to}

In this Section we demonstrate that our method can potentially resolve limitations in calculating mean first passage times observed in other methods of state space truncation, namely sFSP and the finite-buffer method. Both methods require the user to make decisions about the state-space truncation that may introduce variability in the results.
While all methods will converge to the true result as the size of the state space increases, we show our method is less dependent on user defined quantities.
This minimizes additional decision-making on the part of the user that can lead to suboptimal results, especially in a context where the solution of the original system is not known.

\subsection{Comparison to the sFSP method}
A well known state truncation algorithm is known as the \emph{Finite State Projection}  (FSP) method \cite{munsky2006finite}. For a given continuous-time Markov chain, the associated FSP model is restricted to a finite state space.  If the process escapes this truncated space, the state is sent to a designated absorbing state (see Figure \ref{fig:schematic}B). For a fixed time $t$, the probability density function of the original system can be approximated by using the associated FSP model with sufficiently many states. The long-term dynamics of the original system, however,
is not well approximated because the probability flow of the FSP model leaks to the designated absorbing state in the long run.

To fix this limitation of FSP, Gupta et al. proposed the \emph{stationary Finite State Projection} method (sFSP)\cite{gupta2017finite}. This method also projects the original state space onto a finite state space as the FSP method intended to. But sFSP does not create a designated absorbing state, as all outgoing transitions from the projected finite state space are merged to a single state $x^*$ `inside' the finite state space (Figure \ref{fig:schematic}C). The sFSP has been frequently used to estimate the long-term distribution of discrete stochastic models. However, if the size of the truncated state space is not sufficiently large, this method could fail to provide accurate estimation for the first passage time.
To demonstrate this case, we consider the following simple 2-dimensional model. In the network shown in Figure \ref{fig:fspvsslack}A, two $X_1$ proteins are dimerized into protein $X_2$ while $X_1$ is being produced at relatively high rate. The state space of the original model is the full 2-dimensional positive integer grid.  We estimate the time until the system reaches one of the two states indicated in red in Figure~\ref{fig:fspvsslack}C, and we use alternative methods to do this.


\begin{figure*}[!htb]
\centering
\includegraphics[]{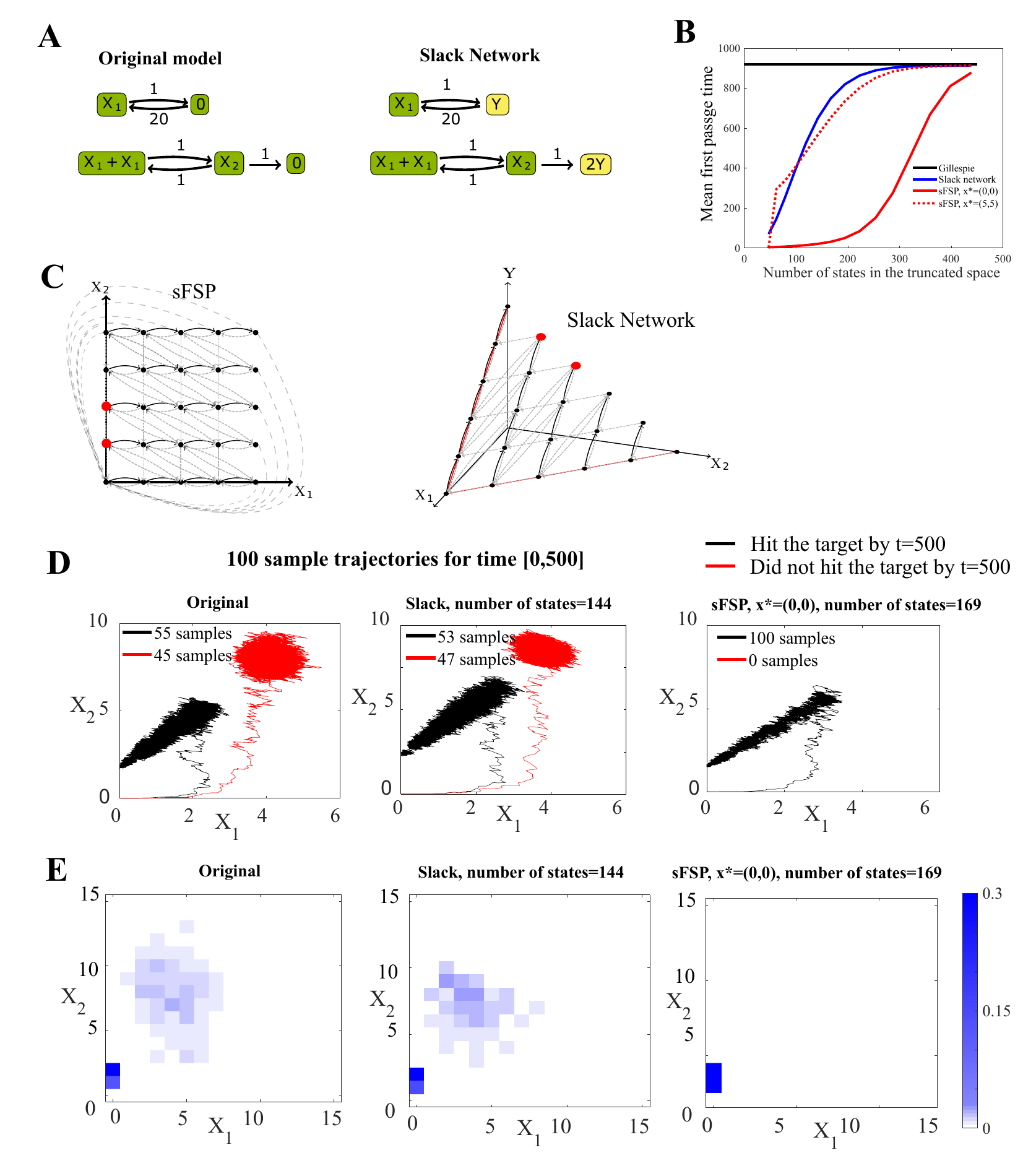}
\caption{Comparison between an sFSP model and a slack network for the same reaction system. A. Original 2-dimensional system and its slack network. B. Mean first passage time as a function of the truncated state size, comparing sFSP and the slack reactant method. C. The state space truncation corresponding to the sFSP method and the slack method for this model. Target states for the first passage time are indicated in red. D. Mean of 100 sample trajectories obtained by Gillespie simulations for the original, the slack system and the sFSP system. The red lines indicate the mean of the trajectories that have touched the target states within [0,500], while the black lines are the mean of the trajectories that did not touch the target space during this time. E. Probability denstiy heat maps of the stochastic processes modeled under the original, the slack system and the sFSP model, respectively, at $t=500$. }\label{fig:fspvsslack}
\end{figure*}

For the sFSP, we project the original state space onto the rectangle by restricting $X_1\le N$ and $X_2\le N$ for some $N>0$, and we fix the origin $(0,0)$ as the designated state $x^*$ (Figure \ref{fig:fspvsslack}C). If the process associated with the sFSP model escapes the rectangle, it transports to the designated state immediately. On the other hand, we also consider a slack network shown in Figure \ref{fig:fspvsslack}B, where we introduce the conservation law $X_1+2X_2+Y=N$ for some $N>0$. Let $\tau=\inf \{t\ge 0: X_1=1 \text{ and } X_2 \in \{1,2\}\}$ be the first passage time we want to estimate. 

By using the inverse matrix of the truncated transition matrix for each method, as detailed in \cite{chou2014first}, we obtain the mean of $\tau$ by using different values of $N$. We also obtain an `almost' true mean of $\tau$ by using $10^5$ Gillespie simulations of the original process. As shown in Figure \ref{fig:fspvsslack}B, the slack network model provides a more accurate mean first passage time estimation for the size of truncation in between 100 and 400 if the designated return state of sFSP is $(0,0)$.

The inaccurate estimate from the sFSP is due to the choice of a return state. The sFSP model escapes the confined space often because the production rate of $X_1$ is relatively high. When it returns back to the origin, it is more likely to visit the target states $\{X_1=1 \text{ and } X_2\in \{1,2\}\}$ than the original process. 

Figure~\ref{fig:fspvsslack} shows that the mean first passage time of this system using sFSP depends significantly on the location of the chosen designated state.  One of the two states is a particularly poor choice for sFSP, but it illustrates the idea that without previous knowledge of the system it can be difficult to know which states will perform well.  

We display the behavior of individual solutions of the original model, the slack network, and the sFSP model in Figure \ref{fig:fspvsslack}D. The trajectory plots show that within the time interval $[0,500]$, almost half of the 100 samples from both the original model and the slack network model stay far away from the target states, while all the 100 sample trajectories from the sFSP model stay close to the target states. We also illustrate this point in Figure \ref{fig:fspvsslack}E with heat maps of the three models at $t=500$. Note that only in the case of the sFSP, the probability densities are concentrated at the target states.

\subsection{Comparison to the finite buffer method}\label{subsec:compare to buffer}
The finite buffer method was proposed to estimate the stationary probability landscape with state space truncation and a novel state space enumeration algorithm \cite{cao2008optimal, cao2016accurate}. 
For a given stochastic model associated with a CRN, the finite buffer method sets inequalities among species such as \eqref{eq:ineq}, so-called buffer capacities. Then at each state $\vec x$, the transition rate of a reaction $\nu\to \nu'$ is set to be zero if at least one of the inequality does not hold at $\vec x+\nu'-\nu$.  We note the algorithm, described in Section~\ref{subsec:algo}, for generating a slack network uses the same inequalities.  Thus,  the finite buffer method and the slack reactant method truncate the state space in the same way.
We have shown, in Section~\ref{subsec:properties of slack} this type of truncation can create `additional' absorbing states. These additional absorbing states change the accessibility between the states which means the mean first passage times cannot be accurately estimated. However, the regular slack systems preserve the network structure of the original network. Hence we are able to prove, as we already noted, that regular slack networks inherit the accessibility of the original network as long as the original network is `weakly reversible' as we will define below.



We demonstrate this disparity between the the finite buffer and slack methods with the following network. Consider the mass-action system  \eqref{eq:ex original} with a fixed initial state $\vec x_0=(a_0,b_0)$. We are interested in estimating the mean first passage time to a target state $\vec x_T=(10,10)$.
Note that the state space is irreducible (i.e., every state is accessible from any other state) as the network consists of unions of cycles. This condition, the union of cycles, is precisely what is meant by weakly reversible.\cite{feinberg1987chemical,august2010solutions} Thus, the original stochastic system has no absorbing state and is accessible to the target state $\vec x_T$.

To use the finite buffer method on this network, we set $2X_A+X_B\le N$ as the buffer capacity, where $X=(X_A,X_B)$ is the associated stochastic process. (Here we choose $N > 30$ so the state space contains the target state $\vec x_T$.) Hence when $X$ satisfies $2X_A+X_B = N$, the reactions $\emptyset \to A$, $B\to A$ and $\emptyset \to B$ cannot be fired as $2X_A+X_B$ exceeds the buffer capacity. We now demonstrate the system has a new absorbing state. By first using reaction $A \to B$, to deplete all $A$, and then $\emptyset \to B$, every state can reach the state $(0,N)$ in finite time with positive probability.
The state $(0,N)$ is absorbing state because no other reactions can occur. Reactions $A \to \emptyset$ and $A\to B$ require at least one $A$ species, and any other reactions lead to states exceeding the buffer capacity. Therefore, the finite buffer method has introduced a new absorbing state not present in the original model so that it is not accessible to $\vec x_T$ with positive probability.

Now, we show the explicit network structure of our slack network formulation will preserve the accessibility of the original system. We consider the same inequality $2X_A+X_B \le N$ as above with $N>30$. We generate the slack network by using the algorithm shown in Section \ref{subsec:algo},
\begin{equation}\label{eq:slack no absorbing}
  \begin{tikzpicture}[baseline={(current bounding box.center)}, scale=0.8,  state/.style={circle,inner sep=2pt}]
  \node[state] (1) at (0,2)  {$2Y$};
  \node[state] (2) at (2,2)  {$A$};
  \node[state] (3) at (0,0)  {$B+Y$};

  \path[->]
    (1) edge[bend left] node   { } (2)
    (2) edge[bend left] node   { } (1)
    (2) edge[bend left] node  { } (3)
    (3) edge[bend left] node  { } (2)
    (1) edge node  { } (3)
    ;
     \end{tikzpicture}
\end{equation}
Note that the associated stochastic process $X=(X_A,X_B,Y)$ admits the conservation relation $2X_A+X_B+Y=N$ implying that $2X_A+X_B \le N$. The state $(0,N,0)$ can not be reached as the only state that is accessible to $(0,N,0)$ is $(0,N-1,2)$, but it violates the conservation law.

As we highlighted in Section \ref{subsec:properties of slack}, slack networks preserve the connectivity of the original network \eqref{eq:ex original}, hence the network \eqref{eq:slack no absorbing} is also weakly reversible. Thus the state space of the stochastic process associated with the slack network is irreducible by Corollary \ref{cor:irreducibility}. This implies that there is no absorbing state and the system is accessible to $\vec x_T$, unlike the stochastic process associated with the finite buffer relation.

\section{Convergence Theorems for slack networks}\label{sec:theorems}

In this section we establish theoretical results on the convergence of properties of a slack network to the original network. 
(Proofs of the theorems below are provided in Appendix \ref{app:proofs}.)
Many of these results rely on theorems from Hart and Tweedie\cite{hart2012convergence} who studied when the probability density function of a truncated Markov process converges to that of the original Markov process. We employ the same idea of their proof to show the convergence of a slack network to the original network. 

By assuming ``uniforming mixing,'' we show the convergence of the stationary distribution of the slack system to the stationary distribution of the original system as the conservation quantity $N$ grows. Furthermore, we show the convergence of mean first passage times for the slack network to the true mean first passage times. In particular, all these conditions hold when there is a Lyapunov function for the original system.

In this Section, assume that a given CRN  $(\S,\C,\Re,\Lambda)$ is well defined for all time $t$ and  let $(\widetilde \S, \widetilde \C,\widetilde \Re, \Lambda^N)$ be an associated slack network obtained with a single conservative quantity $\vec w\cdot X \le N$. We denote by $X$ and $X_N$ the associated stochastic processes the original CRN and the slack network respectively.
We fix the initial state for both systems, i.e., $X(0)=X_N(0)=\vec x_0$ for some $\vec x_0$ and for each $N$. (This means we can only consider slack systems where $N$ is large enough so that $\vec w \cdot \vec x(0) < N$.) Assume that both the original and slack systems are irreducible, and denote by $\mathbb S$ and $\mathbb S_N$ the state spaces for each respectively. (In Section~\ref{sec:acc}, we prove accessibility properties that the slack system can inherit from the original system.) 


Notice that every state in $\mathbb S_N$ satisfies our conservation inequality. That is, for every $\vec x \in \mathbb S_N$ we have $\vec w \cdot \vec x \le N$.
It is possible that $\mathbb S_N=\mathbb S_{N+1}$ for some $N$. For simplicity, we assume that the truncated state space is always enlarged with respect to the conservative quantity $N$, that is $\mathbb S_N \subset \mathbb S_{N+1}$ for each $N$. (For the general case, we could simply consider a subsequence $N_k$ such that $N_k<N_{k+1}$ and $\mathbb S_{N_k} \subset \mathbb S_{N_{k+1}}$ for each $k$.) 


As defined in Section \ref{subsec:CRN}, $\lambda_{\nu \to \nu'} \in \Lambda$ is the intensity of a reaction $\nu \to \nu'$ for the associated stochastic system $X$. We also denote by $\lambda^N_{\tilde \nu\to \tilde \nu '}\in \Lambda^N$ the intensity of a reaction $\tilde \nu \to \tilde \nu'$ for the associated stochastic system $X_N$. Finally we let $p(\vec x,t)$ and $p_N(\vec x,t)$ be the probability density function of $X$ and $X_N$, respectively.  We begin with the convergence of the probability density functions of the slack network to the original network with increasing $N$.

\begin{theorem}\label{thm:conv of probs}
For any $\vec x \in \mathbb S_N$ and $t \geq 0$, we have
\begin{align*}
    \lim_{N\to \infty} |p(\vec x,t)-p_N(\vec x,t)|=0.
\end{align*}
\end{theorem}

A Markov process defined on a finite state space admits a stationary distribution. Hence $X_N$ admits a stationary distribution $\pi_N$. If the slack system satisfies the condition of ``uniform mixing'', that is the convergence rate of $\Vert p_N(\vec x,t)-\pi_N(\vec x) \Vert_1$ is uniform in $N$, then we have the following result:

\begin{theorem}\label{thm:conv of stationary}
Suppose $X$ admits a stationary distribution $\pi$. Suppose further that there exists a positive function $h(t)$, which is independent of $N$, such that $\Vert p_N(\cdot,t)-\pi_N\Vert_1\le h(t)$ and $\dlim_{t\to \infty}h(t)=0$.
Then 
\begin{align*}
\lim_{N\to \infty}\Vert \pi-\pi_N \Vert_1 = 0.
\end{align*}
\end{theorem}

We now consider convergence of the mean first passage time of $X_N$. Recall, we assumed that both stochastic processes have the same initial state $X(0) = X_N(0) = \vec x_0$ and both state spaces $\mathbb S$ and $\mathbb S_N$ are irreducible. 
Hence, for any $K\subseteq \mathbb S$, each state in $\mathbb S_N$ is accessible to $K$ for sufficiently large $N$. 


\begin{theorem}\label{thm:conv of MFPT}
For a subset $K$ of the state space of $X$, let $\tau$ and $\tau_N$ be the first passage times to $K$ for $X$ and to $K \cap \mathbb S_N$ for $X_N$, respectively. 
Assume the following conditions,
\begin{enumerate}
\item $X$ admits a stationary distribution $\pi$, and
\item $\dlim_{N\to \infty}\Vert \pi-\pi_N \Vert_1 = 0$.
\end{enumerate} 
Then for any $t\geq 0$,
\begin{align*}
\lim_{N\to\infty} |P(\tau>t)-P(\tau_N >t)|=0.    
\end{align*}
If we further assume that
\begin{enumerate}
    \item[3.] $E(\tau)<\infty$, and
    \item[4.] there exists $g(t)$ such that $P(\tau_N > t) \le g(t)$ for all $N$ and $\int_0^\infty g(t)dt < \infty$,
\end{enumerate}
then 
\begin{align*}
    \lim_{N\to \infty}|E(\tau)-E(\tau_{N})|=0.
\end{align*}
\end{theorem}

\begin{remark}
To obtain convergence of higher moments of the first passages time, we need only replace  conditions $E(\tau_K)<\infty$ and $\int _0^\infty g(t) dt <\infty$ with 
\begin{align}\label{eq:alter conditions}
E(\tau^n_K)<\infty \quad \text{and} \quad  \int_0^\infty g\left(t^{\frac{1}{n}}\right)dt <\infty, \end{align}
 respectively.
 \end{remark}
We now show that if a Lyapunov function exists for the original system, the conditions in Theorem \ref{thm:conv of stationary} and Theorem \ref{thm:conv of MFPT} hold.
The Lyapunov function approach has been proposed by Meyn and Tweedie \cite{meyn1993stability} and it has been used to study long-term dynamics of Markov processes  \cite{briat2016antithetic, gupta2017finite, anderson2018some, hansen2020existence} especially exponential ergodicity. Gupta et al \cite{gupta2017finite} uses a linear Lyapunov function to show that the stationary distribution of an sFSP model converges to a stationary distribution of the original stochastic model and use the Lyapunov function explicitly compute the convergence rate. In particular, we show Lyapunov functions exist for the examples we consider in Section \ref{sec:practical examples}.

\begin{theorem}\label{thm:Lyapunov}
 Suppose there exists a function $V$ and positive constants $C$ and $D$ such that for all $\vec x$,
 \begin{enumerate}
 \item $V(\vec x) \ge 1$ for all $\vec x\in \mathbb S$,
 \item  $V$ is an increasing function in the sense that $$V(\vec x_{N+1}) \ge V(\vec x_{N})$$ for each $\vec x_{N+1}\in \mathbb S_{N+1}\setminus \mathbb S_{N}$ and $\vec x_{N} \in \mathbb S_{N}$, and
\item $\dsum_{\nu\to \nu' \in \Re} \lambda_{\nu\to \nu'}(\vec x) (V(\vec x+\vec \nu'-\vec \nu)-V(\vec x))$
\vspace{-.3cm}
$$\le -C V(\vec x) +D.$$
 \end{enumerate}
Then the conditions in Theorem \ref{thm:conv of MFPT} hold.

\end{theorem}
\begin{remark}
The conditions \eqref{eq:alter conditions} hold if a Lyapunov function satisfying the conditions in Theorem \ref{thm:Lyapunov} exists. Thus, the convergence of the higher moments of the first passage time also follows. 
\end{remark}

\section{Inheritance of slack networks}\label{sec:acc}
As we showed in Section \ref{subsec:compare to buffer},  not all state space truncations preserve accessibility of states in the original system. (For the example in Section \ref{subsec:compare to buffer}, the truncation created a new absorbing state.) 
Thus, it is desirable to obtain reduced models that are guaranteed to maintain the accessibility of the original system to predetermined target states.
In this Section, we show that under mild conditions, both a regular slack system and an optimized slack system preserve the original accessibility. The proofs of the theorems introduced in this section are in Appendix \ref{app:proofs2}. The key to these results is the condition of weak reversibility.

\begin{definition}
A reaction network is \textbf{weakly reversible} if each connected component of the network is strongly connected. That is,
if there is a path of reactions from a complex $\nu$ to $\nu'$, then there is a path of reactions $\nu'$ to $\nu$.
\end{definition}
\textcolor{black}{We note that the weakly reversible condition applies to the network graph of the CRN. The network graph consists of complexes (nodes) and reactions (edges). 
It is a sufficient condition for irreducibility of the associated mass-action stochastic process.
Indeed, the sufficiency of weak reversibility holds even under general kinetics as long as condition \eqref{eq:standard intensity} is satisfied.\cite{pauleve2014dynamical} Hence irreducibility of a regular slack network follows since it preserves weak reversibility of the original network and the kinetics modeling the regular slack system satisfies \eqref{eq:standard intensity}.}

\begin{cor}\label{cor:irreducibility}
Let $(\S,\C,\Re,\Lambda)$ be a weakly reversible CRN with intensity functions $\Lambda=\{\lambda_{y \to y'}\}$ satisfying \eqref{eq:standard intensity}. Then the state space of the associated stochastic process with a regular slack network $(\widetilde \S,\widetilde \C,\widetilde \Re, \Lambda^N)$ is a union of closed communication classes for any $N$.
\end{cor}

In case the original network is not weakly reversible, we can still guarantee that optimized slack systems have the same accessibility as the original system provided all species have a degradation reaction ($S_i \to \emptyset$).
\begin{theorem}\label{theorem:all out flows}
Let $(\S,\C,\Re,\Lambda)$ be a
reaction network such that $\{S_i \to \emptyset : S_i \in \S\} \subset \Re$. Suppose that the stochastic process $X$ associated with $(\S,\C,\Re,\Lambda)$ and beginning at the point $\vec x_0$ is irreducible. 
Let $X_N$ be the stochastic process associated with an optimized slack network $(\widetilde \S,\widetilde \C,\widetilde \Re, \Lambda^N)$ such that $X_N(0)=\vec x_0$ for every $N$ large enough. Then, for any subset $K$ of the state space of $X$, there exists $N_0$ such that $X_N$ reaches $K$ almost surely for $N\ge N_0$.
\end{theorem}
This theorem follows from the fact that a slack system only differs from the original system when its runs out of slack reactants. However, in an optimized slack system, degradation reactions are allowable with no slack reactants. Hence, our proof of Theorem \ref{theorem:all out flows} relies on the presence of all degradation reactions. 


\textcolor{black}{A slack network may also inherit its stationary distribution from the original reaction system. When the original system admits a stationary distribution of a product form of Poissons under the \emph{complex balance} condition, a slack system inherits the same form of the stationary distribution as well.  
A reaction system is complex balanced if the associated deterministic mass-action system admits a steady state $c^*$, such that
\begin{align*}
    \sum_{\substack{\nu\in \C \\\nu \to \nu'\in \Re}}f_\nu(\vec c^*)= \sum_{\substack{\nu' \in \C \\\nu \to \nu'\in \Re}}f_{\nu'}(\vec c^*),
    \end{align*}
    where $f_\nu(x)=x_1^{\nu_1}\cdots x_d^{\nu_d}$ is the deterministic mass-action rate \cite{feinberg1972complex}.
If a reaction system is complex balanced, then its associated stochastic mass-action system admits a stationary distribution corresponding to a product of Poisson distributions centered at the complex balance steady state.\cite{anderson2010product}. The following lemma show that the complex balancing of the original network is inherited by a regular slack network.}


%
\begin{lem}\label{lem:cb}
Suppose that $(\S,\C,\Re,\Lambda)$ is a reaction network whose mass-action deterministic model admits a complex balanced steady state $c^*$. Then any regular slack network $(\widetilde \S,\widetilde \C,\widetilde \Re,\Lambda^N)$ with slack reactants $Y_1,\cdots, Y_m$ also admits a complex balanced steady state at $\tilde c=(c^*,1,1,\dots,1)^\top$.
\end{lem}
\begin{remark}
\textcolor{black}{Note that a regular slack network also preserves the deficiency of the original network. Deficiency $\delta$ of a reaction network is an index such that
\begin{align*}
    \delta=n-\ell-s,
\end{align*}
where $n$ is the number of the complexes, $\ell$ is the number of connected components and $s$ is the rank of the stoichiometry matrix of the reaction network.
Deficiency characterizes the connectivity of the network structure, and surprisingly it can also determine the long-term behavior of the system dynamics regardless of the system parameters.\cite{feinberg1972complex,horn1972necessary,anderson2010product} A regular slack network and original network have the same number of complexes $n$, and the same connectivity matrix $S$, which implies they have the same number of connected components $\ell$. Furthermore, using the notation from Section \ref{subsec:algo}, the stoichiometry matrices are $\Gamma=CS$ for the original network and $$\widetilde \Gamma = \begin{pmatrix} C \\ U-AC\end{pmatrix} S = \begin{pmatrix} CS \\ -WCS \end{pmatrix}$$ for a slack network, which means that they have the same rank $s$. Together, these imply that the original network and its regular slack network have the same deficiency. }
\end{remark}

Since the complex balancing is inherited with the same steady state values for $X_i$, we have the following stochastic analog of inheritance of the Poissonian stationary distribution for regular slack systems.
\begin{theorem}\label{theorem:relation between two}
Let $X$ be the stochastic mass-action system associated with a complex balanced $(\S,\C,\Re,\Lambda)$ with an initial condition $X(0)=\vec x_0$. Let $X_N$ be the stochastic system associated with a regular slack system $(\widetilde \S,\widetilde \C,\widetilde \Re,\Lambda^N)$ with $X_N(0)=\vec x_0$. 
Then, for the state space $\mathbb S_N$ of $X_N$, there exists a constant $M_N>0$ such that
\begin{align*}
\pi_N(\vec x)=M_N\pi(\vec x)\quad \text{for $\vec x\in \mathbb{S}_N$,}
\end{align*}
where $\pi$ and $\pi_N$ are the stationary distributions of $X$ and $X_N$, respectively.
\end{theorem}

We demonstrate Lemma \ref{lem:cb} and Theorem \ref{theorem:relation between two} with a simple example.
\begin{example}
Consider two networks,
\begin{align}
&X \xrightleftharpoons[2]{1} 0, \label{ex:original}\\ 
&X \xrightleftharpoons[2]{1} Y. \label{ex:slack} 
\end{align}
Let $X$ and $X_N$ be systems \eqref{ex:original} and \eqref{ex:slack}, respectively, where $N$ is the conservation quantity $X_{N}(t)\le N$.
Under mass-action kinetics, the complex balance steady state of 
\eqref{ex:original} is $c^*=2$. Under mass-action kinetics, $\tilde c=(2,1)$ is a complex balance steady state of \eqref{ex:slack}.

Now let $\pi$ and $\pi_N$ be the stationary distribution of $X$ and $X_N$, respectively. 
By Theorem 6.4, Anderson et al \cite{anderson2010product}, $\pi$ is a product form of Poissons such that  
\begin{align*}
\pi(x)=e^{-2}\frac{{c^*}^x}{x!} \quad \text{for each state $x$}.
\end{align*}
By plugging $\pi$ into the chemical master equation \ref{eq:master} of $X_N$ and showing that for each $\vec x$
\begin{align*}
    &\lambda^N_{X\to Y}(\vec x+1)\pi(\vec x+1)+\lambda^N_{Y\to X}(\vec x-1)\pi(\vec x-1)\\
    &=(x+1)\mathbbm{1}_{\{N-x-1\ge 0\}}e^{-2}\frac{2^{x+1}}{(x+1)!}+2\mathbbm{1}_{\{N-x+1\ge 1\}}e^{-2}\frac{2^{x-1}}{(x-1)!}\\
    &=x\mathbbm{1}_{\{N-x\ge 0\}}e^{-2}\frac{2^x}{x!}+2\mathbbm{1}_{\{N-x\ge 1\}}e^{-2}\frac{2^x}{x!}\\
    &=\lambda^N_{X\to Y}(\vec x)\pi(\vec x)+\lambda^N_{Y\to X}(\vec x)\pi(\vec x)
\end{align*}
we can verify that $\pi$ is a stationary solution of the chemical master equation \ref{eq:master} of $X_N$.
Since the state space of $X_N$ is $\{x \in \mathbb Z_{\ge 0} : x \le N\}$, we choose a constant $M_N$ such that
\begin{align*}
    \sum_{0\le x\le N}M_N\pi(x)=1.
\end{align*}
Then $\pi_N=M_N\pi$ is the stationary distribution of $X_N$.
\end{example}

\begin{figure*}[!hbt]
    \centering
    \includegraphics[width = 0.9 \textwidth]{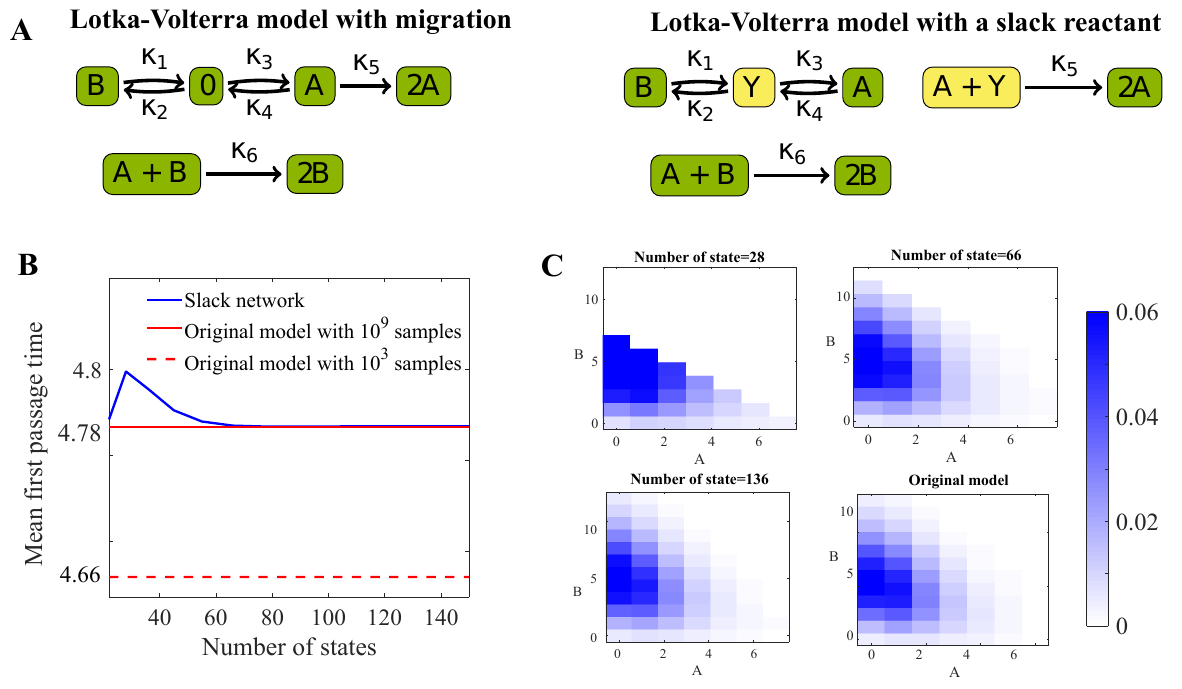}
    \caption{\textbf{Calculating Mean Time to Extinction of a Lotka-Volterra Model with Migration Using Slack Reactants.} A. The reaction network for the Lotka-Volterra model with migration (left).  The Lotka-Volterra model with a slack reactant $Y$ (right). The parameters are $\kappa_1=0.1, \kappa_2=0.1, \kappa_3=0.2, \kappa_4=0.6, \kappa_5=0.2$ and $\kappa_6=0.2$. B. Convergence of the mean time to extinction of the slack network (blue) to the true network (solid red). C. Heatmaps of the probability density at time $1000$ of the original model and the slack system with various truncation size. } 
    \label{fig:GLV}
\end{figure*}

\section{Applications of Slack Networks}\label{sec:practical examples}

In this Section, we demonstrate the utility of slack reactants in computing mean first-passage times for two biological examples. For both examples, we compute the mean first passage time via the matrix inversion approach as shown in Appendix \ref{app:matrix inversion}.

\subsection{A Lotka-Volterra model with migration}

Consider a Lotka-Volterra model with migration shown in Figure \ref{fig:GLV}A. 
In this model, species $A$ is the prey, and species $B$ is the predator. Clearly, the state space this model is infinite $(A,B)$ such that $A \geq 0$, $B \geq 0$. We will use slack reactants to determine the expected time to extinction of either species. More specifically, let $K = \{(A,B): A = 0 \text{ or } B = 0\}$. We will calculate the mean first arrival time to $K$ from an initial condition $(A(0),B(0))$. (In our simulations in Figure \ref{fig:GLV}, we chose $(A(0),B(0)) = (3,3)$.)

To generate our slack network, we choose a conservative bound $w \cdot (A,B)^\top \le N$ with $w=(1,1)$. As we discussed in Section \ref{subsec:properties of slack}, this $w$ minimizes the intrusiveness of slack reactants because the number of reactions $\nu\to \nu'$ such that $(\nu'-\nu)\cdot w >0$ is minimized. By using the algorithm introduced in Section \ref{subsec:algo}, we generate a regular slack network \eqref{eq:regular slack} with a slack reactant $Y$,
\begin{align}
\begin{split}\label{eq:regular slack}
    &B+Y \xrightleftharpoons[\kappa_2]{\kappa_1} 2Y \xrightleftharpoons[\kappa_4]{\kappa_3} A+Y \xrightarrow{\kappa_5} 2A, \\
    &A+B\xrightarrow{\kappa_6} 2B.
    \end{split}
\end{align}
As the slack reactant $Y$ in reactions $B+Y \rightleftharpoons 2Y \rightleftharpoons A+Y$ can be canceled, we further generate the optimized slack network shown in Figure \ref{fig:GLV}A. We let $A(0)+B(0)+Y(0)=N$, which is the conservation quantity of the new network.

Let $\tau$ be the first passage time from our initial condition to $K$. First, we examine the accessibility of the set $K$. Because our reaction network contains creation and destruction of all species (i.e., $B\rightleftharpoons \emptyset \rightleftharpoons A$) the original model is irreducible and any state is accessible to $K$.
Furthermore, Theorem \ref{theorem:all out flows} guarantees that the stochastic model associated with the optimized slack network is also accessible to $K$ from any state.  

Next, by showing there exists a Lyapunov function satisfying the condition of Theorem \ref{thm:Lyapunov} for the original model, we are guaranteed the first passages times from our slack network will converge to the true first passage times. (See Appendix \ref{sec:LyapunovLotkaVoltera} for more detail.)  Therefore, as the plot shows in Figure \ref{fig:GLV}B, the mean first extinction time of the slack network converges to that of the original model, as $N$ increases. \textcolor{black}{The mean first passage time of the original model was obtained by averaging $10^9$ sample trajectories. These trajectories were computed serially on a commodity machine and took $4.6$ hours to run. In contrast, the mean first passage times of the slack systems were computed analytically on the same computer and took at most 13 seconds. Figure \ref{fig:GLV} also shows that using only $10^3$ samples is misleading as the simulation average has not yet converged to the true mean first passage time.} Finally, as expected from Theorem \ref{thm:conv of probs}
the probability density of the slack network converges to that of the original network (see Figure~\ref{fig:GLV}C).

\subsection{Protein synthesis with slow toggle switch}
We now consider a protein synthesis model with a toggle switch, see Figure~\ref{fig:Toggle}A. Protein species $X$ and $Z$ may be created, but only when their respective genes $D^X$ or $D^Z$ are in the active (unoccupied) state, $D^X_0$ and $D^Z_0$. Each protein acts as a repressor for the other by binding at the promoter of the opposite gene and forcing it into the inactive (occupied) state ($D^X_1$ and $D^Z_1$). In this system we consider only one copy of each gene, so that, $D^X_0 + D^X_1 = D^Z_0 + D^Z_1 = 1$ for all time. Thus, we focus primarily on the state space of protein numbers only $(X,Z)$.

The deterministic form of such systems are often referred to as a ``bi-stable switch'' as it is characterized by steady states $(X^*,0)$ and ($X$ ``on'' and $Z$ ``off'') and $(0,Z^*)$ ($X$ ``off'' and $Z$ ``on'').  This stochastic form of toggle switch has been shown to exhibit a multi-modal long-term probability density due to 
switches between these two deterministic states due to rapid significant changes in the numbers of proteins $X$ and $Z$ by synthesis or degradation (depending on the  state of promoters) \cite{al2019multi}.  Figure \ref{fig:Toggle}C shows that the associated stochastic system admits a \emph{tri-modal} long-term probability density. Thus the system transitions from one mode to other modes and, for the kinetic parameters chosen in Figure~\ref{fig:Toggle}, rarely leaves the region $R = \{ (X,Z) | 0 \leq X \leq 30 \text{ or }  0 \leq Z \leq 30\}$.
Significant departures from a stable region of a genetic switch may be associated with irregular and diseased cell fates. As such, the first passage time of this system outside of $R$ 
may indicate the appearance of an unexpected phenotype in a population. 
Because this event is rare, estimating first passage times with direct stochastic simulations, such as with the Gillespie algorithm \cite{gillespie2007stochastic}, will be complicated by the long time taken to exit the region.

As in the previous example, slack systems provide a valuable tool for direct calculation of mean first passage times. In this example, we consider the time a trajectory beginning at state $(X,Z,D^X_0,D^Z_0)=(0,0,1,1)$ enters the target set $K = \{ (X,Z)| X > 30 \text{ and } Z > 30 \} = R^c$ and compute $\tau$, the first passage time to $K$.  

Since the species corresponding to the status of promoters ($D^X_0,D^X_1,D^Z_0$ and $D^Z_1$) are already bounded, we use the conservative bound $X+Z\le N$ to generate a regular slack network in Figure \ref{fig:Toggle}B with the algorithm introduced in Section \ref{subsec:algo}. The original toggle switch model is irreducible (because of the degradation $X\to 0$, $Z\to 0$ and protein synthesis $Z+D^X_0\leftarrow D^X_1$, $X+D^Z_0\leftarrow D^Z_1$ reactions). Moreover by Theorem \ref{theorem:all out flows}, the degradation reactions guarantee that the slack system is also accessible to $K$ from any state. 

As shown in Figure \ref{fig:Toggle}D, the mean first passage time of the slack system 
appears to be converting to approximately $3.171\times 10^9$. To prove that the limit of the slack system is actually the original mean first passage time, we construct a Lyapunov functions satisfying the conditions of Theorem \ref{thm:Lyapunov}. 
See Appendix \ref{sec:LyapunovToggle} for more details about the construction of the Lyapunov function.

\begin{figure*}[!htb]
    \centering
    \includegraphics[width =0.9 \textwidth]{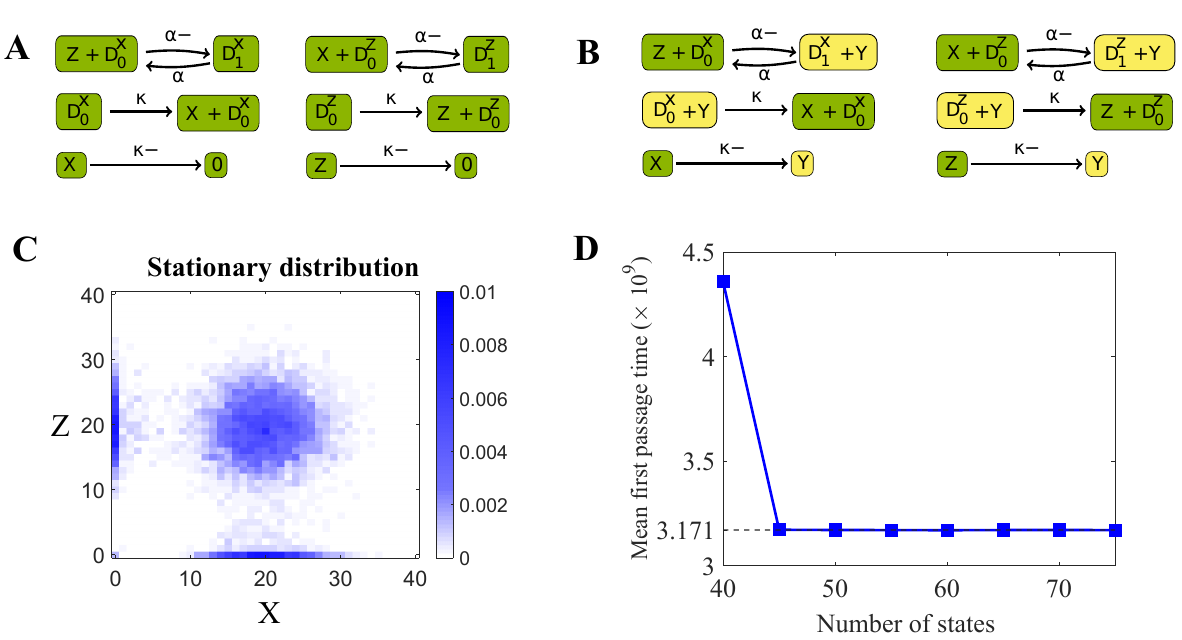}    \caption{\textbf{Calculating Mean First Passage Time of a Slow Toggle Switch Using Slack Reactants.} A. Protein synthesis with slow toggle switch.
    B. The toggle switch model with a slack reactant $Y$. Parameters are $\alpha-=0.1, \alpha=10, \kappa=2000$ and $\kappa-=100$. C. Multimodal long-term probability distribution of the original model. D. 
    Convergence of the first passage time of the slack system with various truncation sizes.
    }
    \label{fig:Toggle}
\end{figure*}

\section{Discussion \& Conclusions}
We propose a new state space truncation method for stochastic reaction networks (see Section~\ref{sec:slack reactant}). In contrast to other methods, such as FSP, sFSP and the finite buffer method, we truncate the state space indirectly by expanding the original chemical reaction network to include slack reactants. The truncation is imposed through user defined conservation relations among species.
The explicit network structure of slack reactants facilitates proofs of convergence (Section~\ref{sec:theorems}) and allows the use of existing software packages to study the slack network itself \cite{kazeroonian2016cerena}.
Indeed, any user-defined choices for conservation laws, conservation amounts, and stochiometric structure, can be used to construct a slack network with our algorithm. We provide guidelines for optimal user choices that can increase the similarity between the slack system and the original model (see Section~\ref{subsec:properties of slack}).

Slack systems can be used to estimate the dynamical behavior of the original stochastic model. In Section \ref{sec:compare to}, we used a simple example to show that the slack method can lead to a better approximation for the mean first passage time than the sFSP method and the finite buffer method. In particular, in Section~\ref{sec:theorems} we provide theorems that show the slack system approximates the probability density and the mean first passage time of the original system. Because slack networks preserve network properties, such as weak reversibility, the slack system is also likely to have the same accessibility to a target state as the original model (see Section~\ref{sec:acc}). In particular, we note that weak reversibility guarantees our slack truncation does not introduce absorbing states.

In Section \ref{sec:acc}, we show that this truncation method is natural in the sense that the stationary distributions of the original and slack system are identical up to multiplication by a constant when the original system is complex balanced. Finally, in Section~\ref{sec:practical examples}, we use slack networks to calculate first passage times for two biological examples.  We expect that this new theoretical framework for state space truncation will be useful in the study of other biologically motivated stochastic chemical reaction systems.

\section*{ACKNOWLEDGMENTS}
This work is partially supported by NSF grants DMS1616233, DMS1763272, BCS1344279,  Simons Foundation grant 594598 and the by the Joint DMS/NIGMS Initiative to Support Research at the Interface of the Biological and Mathematical Sciences (R01-GM126548).

\section*{Data Availability}
The data that support the findings of this study are available from the corresponding author upon reasonable request.
\appendix

\section{First passage time for Markov processes with finitely many states}\label{app:matrix inversion}
The Markov chain associated with a slack network has always a finite state space. There are many different methods to analytically derive the mean first passage time of a Markov chain with a finite state space \cite{hunter2018computation, hunter1980generalized,yvinec2012first,chou2014first}. In this paper, we use the method of Laplace transform which is also used in \cite{yvinec2012first}. 

For a continuous time Markov process, let $\mathbb S$ be the finite state space and let $Q$ be the transition rate matrix, i.e. $Q_{ij}$ is the transition rate from state $i$ to state $j$ if $i\neq j$ and $Q_{ii}=-\sum_{j \neq i} Q_{ij}$.

For a subset $K=\{i_1,i_2,\dots,i_k\} \subset \mathbb S$, we define a \emph{truncated transition matrix} $Q_K$ that is obtained by removing the $i_j$ th row and column from $Q$ for $j=1,2,\dots,k$. Then the mean first passage time to set $K$ starting from the $i$-th state is
the $i$-th entry of $-Q_K^{-1}\mathbf 1$, where $\mathbf 1$ is a column vector with each entry $1$.

\section{Proofs of Convergence Theorems}\label{app:proofs}
In this section, we provide the proofs of the theorems introduced in Section \ref{sec:theorems}. We use the same notations and the same assumptions as we used in Section \ref{sec:theorems}. 

\textbf{Proof of Theorem \ref{thm:conv of probs}:}
We employ the main idea shown in the proof of Theorem 2.1 in Hart and Tweedie \cite{hart2012convergence}. Let a state $\vec x$ and time $t$ be fixed. We consider large enough $N$ so that $\vec x \in \mathbb S_{N-1}$.

We use an FSP model on $\mathbb S_{N-1}$ of the original system $X$ with the designated absorbing state $\vec x^* \in \mathbb S_{N-1}^c$. Let $p^{\text{FSP}}_{N-1}$ be the probability density function of this FSP model.

Let $T_N$ be the first time for $X_N$ to hit $\mathbb S_N \setminus \mathbb S_{N-1}$.
We generate a coupling of $X_N$ and the FSP model restricted on $\mathbb S_{N-1}$ as they move together by $T_N$  and they move independently after $T_N$. Then $ p^{\text{FSP}}_{N-1}(\vec x,t)= P(X_N(t)=\vec x, T_N>t)$ for $\vec x\in \mathbb S_{N-1}$ because  the FSP model has stayed in $\mathbb S_{N-1}$ if and only if $X_N$ has never touched $\mathbb S_N\setminus \mathbb S_{N-1}$.  Thus 
\begin{align}
    p_N(\vec x,t)&=P(X_N(t)=\vec x,t<T_N)+P(X_N(t)=\vec x, t \ge T_N) \notag \\
    & = p^{\text{FSP}}_{N-1}(\vec x,t) + P(X_N(t)=\vec x, T_N\le t) \label{eq:conv of prob 1}\\
    &\ge  p^{\text{FSP}}_{N-1}(\vec x,t). \notag
\end{align}

Furthermore
\begin{align*}
    &P(X_N(t)=\vec x, t \ge T_N) \le P(t \ge T_N) \\
    &= p^{\text{FSP}}_{N-1} (\vec x^*,t)=1-p^{\text{FSP}}_{N-1} (\mathbb S_{N-1},t),
\end{align*} 
where we used the fact that after $T_N$, the FSP process is absorbed at $\vec x^*$.
Thus
\begin{align}\label{eq:conv of prob final}
  p^{\text{FSP}}_{N-1}(\vec x,t)\le    p_N(\vec x,t) \le p^{\text{FSP}}_{N-1}(\vec x,t)+ 1-p^{\text{FSP}}_{N-1} (\mathbb S_{N-1},t).
\end{align}
Note that 
\begin{align*}
    p^{\text{FSP}}_{N-1}(\vec x,t) = P(X(t)=\vec x, t<T_N) \le p(\vec x,t).
\end{align*}
Since $T_N$ increases to  $\infty$ almost surely, as $N$ increases, 
$p^{\text{FSP}}_{N-1}(\vec x,t)$ monotonically increases in $N$ and converges to $p(\vec x,t)$, as $N\to \infty$ for each $\vec x\in \mathbb S_N$.
Then by using the monotone convergence theorem, the term $$p^{\text{FSP}}_{N-1} (\mathbb S_{N-1},t)=\sum_{x\in \mathbb S}p^{\text{FSP}}_{N-1} (x,t)\mathbbm{1}_{\{x \in \mathbb S_{N-1}\}}$$
converges to $p(\mathbb S,t)=1$.
Therefore, by taking $\dlim_{N\to \infty}$ in both sides of \eqref{eq:conv of prob final} the result follows. \hfill $\square$

\textbf{Proof of Theorem \ref{thm:conv of stationary}:}
This proof is a slight generalization of the proof of Theorem 3.3 in Hart and Tweedie \cite{hart2012convergence}.
Since the convergence of $\Vert p_N(\cdot,t)-\pi_N\Vert_1$ is independent of $N$, for any $\epsilon>0$, we choose sufficiently large $t_0$ such that 
\begin{align*}
    \Vert p(\cdot,t_0)-\pi(t_0)\Vert_1 \le \epsilon \quad \text{and} \quad \Vert p_N(\cdot,t_0)-\pi_N(t_0)\Vert_1 \le \epsilon
\end{align*}
for all $N$. Then by using the triangle inequalities
\begin{align}
    \Vert \pi-\pi_N\Vert_1 &\le \Vert p(\cdot,t_0) - \pi \Vert_1 + \Vert p_N(\cdot,t_0)-\pi_N\Vert_1 \notag \\
    &\ \ + \Vert p(\cdot,t_0)-p_N(\cdot, t_0)\Vert_1\notag \\
    &\le \Vert p(\cdot,t_0) - \pi \Vert_1 + \Vert p_N(\cdot,t_0)-\pi_N\Vert_1 \notag \\
    &\ \ + \Vert p(\cdot,t_0)-p^{\text{FSP}}_{N-1}(\cdot,t_0) \Vert_1 \notag \\
    & \ \ +\Vert p^{\text{FSP}}_{N-1}(\cdot,t_0)-p_N(\cdot,t_0) \Vert_1 \notag \\
    \begin{split}\label{eq:conv of mfpt final}
    & \le 2\epsilon + \sum_{\vec x\in \mathbb S} |p(\vec x ,t_0)-p^{\text{FSP}}_{N-1}(\vec x,t_0)|\\
    &\ \ + \sum_{x\in \mathbb S}|p^{\text{FSP}}_{N-1}(\vec x,t_0)-p_N(\vec x,t_0)|.
    \end{split}
\end{align}
Note that as we mentioned in the proof of Theorem \ref{thm:conv of probs}, we have monotone convergence of $p^{\text{FSP}}_{N-1}(\vec x,t_0)$  to $p(\vec x,t_0)$ for each $\vec x \in \mathbb S_{N-1}$, as $N \to \infty$. 
Hence by the monotone convergence theorem, the first summation in \eqref{eq:conv of mfpt final} goes to zero, as $N\to \infty$. Note further that from \eqref{eq:conv of prob 1} we have $$|p^{\text{FSP}}_{N-1}(\vec x,t_0)-p_N(\vec x,t_0)|= P(X_N(t)=\vec x, T_N\le t_0).$$ Hence the second summation in \eqref{eq:conv of mfpt final} satisfies
\begin{align*}
    \sum_{x\in \mathbb S}|p^{\text{FSP}}_{N-1}(\vec x,t_0)-p_N(\vec x,t_0)| &\le
    \sum_{x\in \mathbb S_{N-1}}|p^{\text{FSP}}_{N-1}(\vec x,t_0)-p_N(\vec x,t_0)|\\
    & \ \ + p^{\text{FSP}}_{N-1}(x^*,t_0) \\
    & \ \ + P(X_N(t)\in \mathbb S_N\setminus \mathbb S_{N-1}).
\end{align*}
Note that by \eqref{eq:conv of prob 1}
\begin{align*}
&\sum_{x\in \mathbb S_{N-1}}|p^{\text{FSP}}_{N-1}(\vec x,t_0)-p_N(\vec x,t_0)|\\
&=\sum_{x\in \mathbb S_{N-1}}P(X_N(t_0)=\vec x,t_0\ge T_N) = P(t_0\ge T_N)    
\end{align*}
 Furthermore $p^{\text{FSP}}_{N-1}(x^*,t_0)=P(t_0\ge T_N)$ and $P(X_N(t)\in \mathbb S_N\setminus \mathbb S_{N-1})=P(T_N<t_0)$. 
Hence $\sum_{x\in \mathbb S}|p^{\text{FSP}}_{N-1}(\vec x,t_0)-p_N(\vec x,t_0)|\to 0$, as $N\to \infty$ because $T_N\to \infty$ almost surely, as $N\to \infty$. 

Consequently, we have 
\begin{align*}
    \lim_{N\to \infty} \Vert \pi-\pi_N\Vert_1  \le 2\epsilon.
\end{align*}
Since we choose an arbitrary $\epsilon$, the completes the proof.
 \hfill $\square$

In order to prove Theorem \ref{thm:conv of MFPT}, we consider an `absorbing' Markov process associated with $X$ and $X_N$.
Let $\bar X$ and $\bar X_N$ be Markov processes such that
\begin{align*}
    \bar X(t)=\begin{cases}
    X(t) &\text{ if $t<\tau$,}\\
    X(\tau) &\text { if $t\ge \tau$}.
    \end{cases}\\
    \bar X_N(t)=\begin{cases}
    X_N(t) &\text{ if $t<\tau_N$,}\\
    X_N(\tau_N) &\text { if $t\ge \tau_N$}.
    \end{cases}
\end{align*}
That is, $\bar X$ and $\bar X_N$ are coupled process to $X$ and $X_N$, respectively. Furthermore they are absorbed to $K$ once the coupled process ($X$ for $\bar X$ and $X_N$ for $\bar X_N$) visits the set $K$. These coupled process have the following relation.

\begin{lem}\label{lem1}
Let $t \geq 0$ be fixed and $\epsilon > 0$ be arbitrary.
Suppose $X$ admits a stationary distribution $\pi$.
Suppose further that $\dlim_{N\to \infty} \Vert \pi - \pi_N\Vert_1=0$. Then there exists a finite subset $\bar K$ such that $P(\bar X(t) \in \bar K^c) < \epsilon$ and $P(\bar X_N(t) \in \bar K^c) < \epsilon$ for sufficiently large $N$.
\end{lem}
\begin{proof}
Since the probabilities of $\bar X$ and $\bar X_N$ are leaking to $K$, we have for each $t$ and for each $\vec x \in K^c$ 
\begin{align}
\begin{split}\label{eq:two probs bounded}
&P(\bar X(t)=\vec x) \le P( X(t)=\vec x), \quad \text{and} \\
&P(\bar X_N(t)=\vec x) < P( X_N(t)=\vec x).
\end{split}
\end{align}
This can be formally proved as for each $x\in K^c$
\begin{align*}
    P(X(t)=\vec x)&=P(X(t)=\vec x, \tau > t)+P(X(t)=\vec x, \tau \le t) \\
    &= P(\bar X(t)=\vec x)+P(X(t)=\vec x, \tau^K \le t)\\
    &\ge P(\bar X(t)=\vec x).
\end{align*}
In the same way, we can prove that $P(\bar X_N(t)=\vec x) < P( X_N(t)=\vec x)$ for each $\vec x\in K^c$.

Let $\epsilon'>0$ be arbitrary. Then $\dlim_{N\to \infty} \Vert \pi - \pi_N\Vert_1=0$ implies that for any subset $U$, we have 
\begin{align}\label{eq:two pi's are close}
\pi_N(U)\le \pi(U)+\epsilon' \quad \text{and}\quad     \pi(U)\le \pi_N(U)+\epsilon',
\end{align}
for sufficiently large $N$.
We use this property and \eqref{eq:two probs bounded} combined with the `monotonicity' of the chemical master equation to show the result. 

First of all, note that we are assuming $X(0)=X_N(0)=\vec x_0$. Hence there is a constant $\gamma_1>0$ such that $$P(X(0)=\vec x)=\mathbbm{1}_{\{\vec x=\vec x_0\}} \le \gamma \pi(\vec x)$$ for all $\vec x$. Moreover by choosing sufficiently small $\epsilon'$ in \eqref{eq:two pi's are close},  we have the minimum $\min_N \pi_N(\vec x_0) \ge \pi(\vec x_0)-\epsilon'>0$. Thus there exists $\gamma_2$ such that 
$$P(X_N(0)=\vec x) =\mathbbm{1}_{\{\vec x=\vec x_0\}} \le \gamma_2 \pi_N(\vec x),$$    
for all sufficiently large $N$ and for each $\vec x$. Note that the chemical master equation is a system of ordinary differential equations, and $P(X(t)=\cdot)$ and $\pi$ are the solution to the system with the initial conditions $P(X(0)=\cdot)$ and $\pi$. Hence the monotonicity of a system of ordinary differential equations, it follows that 
\begin{align}\label{eq:p bound}
P(X(t)=x) \le \gamma_1 \pi(x) \quad \text{for any $x$ and $t$}. 
\end{align}
In the same way, for any $N$ it follows that  
\begin{align}\label{eq:pn bound}
P(X_N(t)=x) \le \gamma_2 \pi_N(x)\quad \text{for any $x$ and $t$}.    
\end{align}
 The detailed proof about the monotonicity of the chemical master equation is shown in Lemma 3.2 of Enciso and Kim. \cite{enciso2019constant}

Secondly, there exists a finite set $\bar K$ such that $\pi(\bar K^c)< \epsilon'$ because $\pi$ is a probability distribution,  Furthermore  by \eqref{eq:two pi's are close}, $\pi_N(\bar K^c) \le \pi(\bar K^c)+\epsilon'< 2\epsilon'$ for any sufficiently large $N$.

Finally we choose $\epsilon'=\frac{\epsilon}{2\gamma}$, where $\gamma=\max\{\gamma_1,\gamma_2\}$. Then by summing up \eqref{eq:two probs bounded}, \eqref{eq:p bound} and \eqref{eq:pn bound} over $x\in \bar K^c$, the result follows.
\end{proof}

\vspace{0.5cm}

We prove Theorem \ref{thm:conv of MFPT} by using the `absorbing` Markov processes $\bar X$ and $\bar X_N$ coupled to $X$ and $X_N$, respectively.

\textbf{Proof of Theorem \ref{thm:conv of MFPT}:}
We first break the term $P(\tau>t)-P(\tau_N>t)$ to show that
$$\lim_{N\to \infty}|P(\tau>t)-P(\tau_{N}>t)|=0.$$

Let $\epsilon>0$ be an arbitrarily small number. Let $\bar K$ be a finite set we found in Lemma \ref{lem1}.
Then by  using triangular inequalities, 
\begin{align}\label{eq:P of Xbar's conv}
    \begin{split}
    &|P(\tau > t)-P(\tau_N > t)|\\
    &=|P(\overline X(t) \in K^c)-P(\overline X_N(t) \in K^c)| \\
     &\le \sum_{\vec x\in K^c \cap \bar K} |P(\overline X(t)=\vec x)-P(\overline X_N(t) =\vec x)|\\
     &+\sum_{\vec x\in K^c \cap \bar K^c} |P(\overline X(t)=\vec x)-P(\overline X_N(t) =\vec x)| \\
     &\le \sum_{\vec x\in K^c \cap \bar K} |P(\overline X(t)=\vec x)-P(\overline X_N(t) =\vec x)| \\
     & \ \ + P(\overline X(t) \in \bar K^c) + P(\overline X_N (t) \in \bar K^c) \\
     &\le \sum_{\vec x\in K^c \cap \bar K} |P(\overline X(t) =\vec x)-P(\overline X_N(t) =\vec x)| + 2\epsilon. 
\end{split}
\end{align}

Note that by the same proof of Theorem  \ref{thm:conv of probs}, we have the convergence
\begin{align*}
    \lim_{N\to \infty}|P(\overline X(t) =\vec x)-P(\overline X_N(t) =\vec x)| = 0 \quad \text{for each $\vec x \in \mathbb S$}.
\end{align*}
Since the summation $\sum_{\vec x\in K^c \cap \bar K}$ is finite, we have that by taking $N\to \infty$ in \eqref{eq:P of Xbar's conv}
\begin{align}\label{eq:PP conv}
    \lim_{N\to \infty}|P(\tau^K>t)-P(\tau^K_N>t)|=0+2\epsilon.
\end{align}

Now, to show the convergence of the mean first passage times, note that 
\begin{align*}
    |E(\tau)-E(\tau_N)|\le \int_0^\infty|P(\tau>t)-P(\tau_N>t)|dt,
\end{align*}
where the integrand is bounded by $P(\tau_K>t)+g(t)$. 

Condition (iii) in Theorem \ref{thm:conv of MFPT} implies that $E(\tau_k)<\infty$ and $\int_0^\infty g(t)dt < \infty$. Hence the dominant convergence theorem and \eqref{eq:PP conv} imply that 
\begin{align*}
    \lim_{N\to \infty}|E(\tau)-E(\tau_N)|=0.
\end{align*}

 \hfill $\square$

The existence of a special Lyapunov function ensures that the conditions in Theorem \ref{thm:conv of MFPT} hold.
In order to use the absorbing Markov processes, as we did in the previous proof, we define a Lyapunov function for $\overline X$ based on a given Lyapunov function $V$. Let $\bar \lambda_{\nu \to \nu'}$ and $\bar \lambda^N_{\nu \to \nu'}$ denote the intensity of a reaction $\nu \to \nu'$ for $\overline X$ and $\overline{X}_N$, respectively.
For a given function $V$ such that $V(\vec x)\ge 1$ for any $\vec x \in \mathbb S$, we define $\overline V$ such that
\begin{align*}
    \overline V(\vec x) =\begin{cases}
    V(\vec x) \quad &\text{if $\vec x \in K^c$}\\
    1\quad &\text{if $\vec x \in K$},\\
    \end{cases}
\end{align*}
 so that $V(\vec x)\ge V(\vec x)$ for any $\vec x$.

Note that $\bar \lambda_{\nu\to \nu'}(x)=\lambda_{\nu\to \nu'}(x)$ if $x\in K^c$.
Hence for each $\vec x\in K^c$,
\begin{align}
\begin{split}\label{eq: relation between v and v bar}
    &\sum_{\nu\to \nu' \in \Re}\bar \lambda_{\nu\to \nu'}(\vec x)(\overline V(\vec x + \vec \nu' - \vec \nu)-\overline V(\vec x))\\
    &\le\sum_{\nu\to \nu' \in \Re}\lambda_{\nu\to \nu'}(\vec x)(V(\vec x + \vec \nu' - \vec \nu)- V(\vec x))
    \end{split}
\end{align}
because $V(\vec x)=\overline V(\vec x)$ and $V(x+\nu'-\nu)\ge\overline V(x+\nu'-\nu)$ for every reaction $x+\nu'-\nu$.
Moreover, Since $\bar \lambda_{\nu\to \nu'}(x)=0\le \lambda_{\nu  \to \nu'}(x)$ if $x\in K$ by the definition of $\overline X$ for each reaction $\nu\to \nu'$.
Hence, for each $x\in K$
\begin{align}\label{eq: relation between v and v bar2}
    \sum_{\nu\to \nu' \in \Re}\bar \lambda_{\nu\to \nu'}(\vec x)(\overline V(\vec x + \vec \nu' - \vec \nu)-\overline V(\vec x))=0.
    \end{align}
From \eqref{eq: relation between v and v bar} and \eqref{eq: relation between v and v bar2}, therefore, we conclude that for any $x$
\begin{align}
\begin{split}\label{eq:Lyapunov condition of bar x}
     &\sum_{\nu\to \nu' \in \Re}\bar \lambda_{\nu\to \nu'}(\vec x)(\overline V(\vec x + \vec \nu' - \vec \nu)-\overline V(\vec x))\\
     &\le \begin{cases}
     -CV(\vec x)+D \quad &\text{if $\vec x \in K^c$},\\
     0 &\text{if $\vec x \in K$}
     \end{cases}\\
     &\le
     -C'\bar V(x) +D',
\end{split}
\end{align}
where $C'=C$ and $D'=\max\{C+1, D\}$.

\textbf{Proof of Theorem \ref{thm:Lyapunov}: }
In this proof, we show that all the conditions in Theorem \ref{thm:conv of MFPT} are met by using the given function $V$. First, we show that $X$ admits a stationary distribution. This follows straightforwardly by Theorem 3.2 in Hard and Tweedie \cite{hart2012convergence} because condition 3 in Theorem \ref{thm:Lyapunov} means that $V$ is a Lyapunov function for $X$.

The condition basically means that every outward reaction $\nu\to \nu'$ (i.e. $\vec x+\nu'-\nu \in \mathbb S^c_N$ and $\vec x\in \mathbb S_N$) in $\Re$ gives a non-negative drift as $V(x+\nu'-\nu)-V(x)\ge 0$. We use condition 2 in Theorem \ref{thm:Lyapunov} to show that $V$ is also a Lyapunov function for $X_N$ for any N.

We denote by $\lambda^N_{\tilde \nu\to \tilde \nu'}$ the intensity of a reaction $\tilde \nu\to \tilde \nu'$ in $(\widetilde \S,\widetilde \C, \widetilde \Re,\Lambda^N)$. Suppose that a reaction $\tilde \nu \to \tilde \nu'\in \widetilde \Re$ is obtained from $\nu\to \nu' \in \Re$ by adding a slack reactant. That is, $q(\tilde \nu'-\tilde \nu)=\nu'-\nu$, where $q$ is a projection function such that $q(x_1,\dots,x_d,x_d+1,\dots,x_{d+r})=(x_1,\dots,x_d)$. Then, by the definition \eqref{eq:slack intensity} of the intensity in a slack network, we have $\lambda^N_{\tilde \nu\to \tilde \nu'}(\vec x)=0$ when $\vec x + q(\tilde \nu'-\tilde \nu) \not \in \mathbb S_N$, because the $X_N$ is confined in $\mathbb S_N$. 
Furthermore, by condition 2 in Theorem \ref{thm:Lyapunov}, $V(x+\vec \nu'-\vec \nu) \ge V(x)$ when $\vec x+q(\tilde \nu'-\tilde \nu) \not \in \mathbb S_N$ because this means that $x+\vec \nu'-\vec \nu \in \mathbb S_M$ for some $M>N$. This implies if $x+\vec \nu'-\vec \nu\not \in  \mathbb S_N$
\begin{align}
\begin{split}\label{eq:relation V and VN}
    0& = \lambda^N_{\tilde \nu \to \tilde \nu'}(\vec x)(V(x+\vec \nu'-\vec \nu)-V(x))\\
    &\le \lambda_{\nu\to \nu'}(\vec x)(V(x+\vec \nu'-\vec \nu)-V(x)).
    \end{split}
\end{align}
In case $\vec x+q(\tilde \nu'-\tilde \nu) \in \mathbb S_N \subseteq \mathbb S$, we have that $\lambda^N_{\tilde \nu \to \tilde \nu'}(\vec x)=\lambda_{\nu \to \nu'}(\vec x)$ by the definition \eqref{eq:slack intensity}. 
Hence by condition 3 and \eqref{eq:relation V and VN}  we have for any $\vec x$ 
\begin{align}
\begin{split}\label{eq:relation2 V and VN }
   &\sum_{\tilde \nu \to \tilde \nu' \in \Re}\lambda^N_{\tilde \nu \to \tilde \nu'}(\vec x)(V(\vec x + \vec \nu' - \vec \nu)- V(\vec x))\\
   &\le  \sum_{\nu\to \nu' \in \Re}\lambda_{\nu\to \nu'}(\vec x)(V(\vec x + \vec \nu' - \vec \nu)- V(\vec x))\\
   &\le -CV(x)+D.
   \end{split}
\end{align}
Thus $V$ is also a Lyapunov function for $X_N$.
Hence Theorem 6.1 (the Foster-Lyapunov criterion for exponential ergodicity) in Meyn and Tweedie \cite{meyn1993stability} implies that for each $N$ there exist $\beta>0$ and $\eta>0$, which are only dependent of $C$ and $D$, such that
\begin{align*}
    \Vert p_N(\cdot,t)-\pi_N\Vert_1 \le \beta V(\vec x_0) e^{-\eta t}.  
\end{align*}
This guarantees that the condition in Theorem \ref{thm:conv of stationary} holds with $h(t)=\beta V(\vec x_0) e^{-\eta t}$. Hence we have $\dlim_{N\to \infty} \Vert \pi - \pi_N \Vert_1 =0$ by Theorem \ref{thm:conv of stationary}.

Now to show that the first passage time $\tau$ has the finite mean, we apply the Foster-Lyapunov criterion to $\bar X$. Since \eqref{eq:Lyapunov condition of bar x} meets the conditions of Theorem 6.1 in Meyn and Tweedie \cite{meyn1993stability}, the probability of $\bar X$ converges in time to its stationary distribution exponentially fast. That is, for any subset $U$, there exist $\bar \beta>0$ and $\bar \eta >0$ such that
\begin{align*}
    \Vert P(\bar X(t) \in U)-\bar \pi(U) \Vert \le   \beta V(\vec x_0)e^{-\eta t},
\end{align*}
where $\bar \pi$ is the stationary distribution of $\overline X$.
This, in turn, implies that
\begin{align}
\begin{split}\label{eq:derive finite moment}
    E(\tau_K)&=
    \int_0^\infty P(\tau > t)dt\\
    &=\int_0^\infty P(\bar X(t) \in K^c) dt\\
&\le \int_0^\infty |P(\bar X(t) \in K^c)-\bar \pi(K^c)|+\bar \pi(K^c) dt\\
&\le \int_0^\infty \beta V(\vec x_0)e^{-\eta t} dt < \infty,
\end{split}
\end{align}
where the second inequality follows as $\bar \pi(K^C)=0$, which is because $\bar X$ is eventually absorbed in $K$ as the original process $X$ is irreducible and closed. 

Finally we show that for any $N$, there exists $g(t)$ such that $P(\tau_{N}<t) < g(t)$ and $\int_0^\infty g(t)dt <\infty$ for the first passage time $\tau_{K,N}$. By the same reasoning we used to derive \eqref{eq:Lyapunov condition of bar x}, we also derive by using \eqref{eq:relation2 V and VN } that
\begin{align}
\begin{split}\label{eq:Lyapunov condition of bar x N}
    &\sum_{\tilde \nu\to \tilde \nu' \in \Re}\bar \lambda^N_{\tilde \nu\to \tilde \nu'}(\vec x,y)(\overline V(\vec x +  \nu' - \nu )-\overline V(\vec x))\\
     & \le  -C' \overline V(\vec x) +D'.
\end{split}
\end{align}
Hence in the same way as used for \eqref{eq:derive finite moment}, we have the exponential ergodicty of $\overline X_N$ by Theorem 6.1 in Meyn and Tweedie \cite{meyn1993stability}, and then we derive that
\begin{align*}
    P(\tau_{N}>t)&=P(\overline X_N(t) \in K^c) \\
&\le |P(\overline X(t) \in K^c)-\bar \pi _N(K^c)|+\bar \pi_N(K^c) \\
&\le \beta V(\vec x_0)e^{-\eta t} ,
\end{align*}
where $\bar \pi_N$ is the stationary distribution of $\overline X_N$ which also has zero probability in $K^C$ as $\overline X_N$ is eventually absorbed in $K$. Since $\beta$ and $\eta$ only depend on $C'$ and $D'$, we let $g(t)=\beta V(\vec x_0)e^{-\eta t} $ that satisfies that $P(\tau_{K,N}>t)\le g(t)$ for any $N$ and $\int_0^\infty g(t) dt < \infty$. \hfill $\square$

\section{Proofs of Accessibility Theorems}\label{app:proofs2}
In this section, we prove Theorem \ref{theorem:all out flows}. We begin with a necessary lemma.
In the following lemma, for a given reaction system $(\S, \C,\Re, \Lambda)$, we generate a slack system 
$(\widetilde S, \widetilde \C, \widetilde \Re, \Lambda^N)$ admitting a single slack reactant $Y$. We denote by $\vec w$ the vector for which the slack system admits a the conservation bound $\vec w\cdot x \le N$ for each state $x$ of the slack system. We also let $c$ be the maximum stochiometry coefficient of the slack reactant $Y$ in the slack network. 
Finally note that $\widetilde \C$ and $\C$ has the one-to-one correspondence as every complex $\tilde \nu \in \widetilde \C$ is obtained by adding the slack reactant $Y$ to a complex $\nu \in \C$ with the stochiometric coefficient $u_{\nu}$. That is, $\tilde \nu=(\nu,u)$ for some $\nu\in \C$ and $u_{\nu}$. Hence as the proof of Theorem \ref{thm:Lyapunov}, we let $q$ be a one-to-one projection function such that $q(\tilde \nu)=\nu$ where $\tilde \nu=(\nu,u)$ for some $\nu \in \C$ and $u_{\nu}$.
\begin{lem}\label{lem2}
For a reaction system $(\S,\C,\Re,\Lambda)$, let $(\widetilde S, \widetilde \C, \widetilde \Re,\Lambda^N)$ be a slack system with a single slack reactant $Y$ and a conservation vector $\vec w$. Let $c$ be the maximum stochiometric coefficient of the slack reactant $Y$ in $(\widetilde S, \widetilde \C, \widetilde \Re, \Lambda^N)$. Then if $c \le N-(\vec w\cdot \vec x)$ for a state $\vec x$, then
\begin{align*}
    \lambda^N_{\tilde \nu\to \tilde \nu'}(x) >0 \quad \text{if only if} \quad \lambda_{\nu\to \nu'}(x)>0,
\end{align*}
where $\tilde \nu \to \tilde \nu' \in \widetilde \Re$ and $\nu \to \nu' \in \Re$ are reactions such that
$q(\tilde \nu)=\nu$ and $q(\tilde \nu)= \nu'$.
\end{lem}
\begin{proof}
Let $\tilde \nu \in \widetilde \C$ such that $\tilde \nu=(\nu,u_{\nu})$ with the stoichiometric coefficient $u_{\nu}$ of $Y$.
By the definition of $\lambda^N_{\tilde \nu\to \tilde \nu'}$ shown in \eqref{eq:slack intensity}, 
\begin{align*}
    \lambda^N_{\tilde \nu\to \tilde \nu'}(x) >0 \quad \text{if only if} \quad \lambda_{\nu\to \nu'}(x)>0,
\end{align*}
so long as $y=N-\vec w \cdot \vec x \ge u_{\nu}$. Since $c=\max_{\nu \in \C} \{u_{\nu}\}$, the result follows.
\end{proof}

\vspace{0.5cm}

\textbf{Proof of Theorem \ref{theorem:all out flows}: }
We denote by $\mathbb S$ the irreducible and closed state space of $X$ such that $X(0)=\vec x_0$. We also denote by $\mathbb S_{N}$ the communication class of $X_N$ such that $x_0 \in \mathbb S_{N}$ for each $N$. Then $\mathbb S_N\subseteq \mathbb S$, and the irreducibility of $\mathbb S$ guarantees that $\cup_{i=1}^\infty \mathbb S_i = \mathbb S$. Therefore there exists $N_1$ such that $\mathbb S_{N,\vec x_0} \cap K\neq \emptyset$ if $N\ge N_1$. This implies that it is enough to show that $\mathbb S_{N,\vec x_0}$ is closed for $N$ large enough because $\mathbb S_{N,\vec x_0}$ is a finite subset. To prove this, we claim  that there exists $N_0$ such that if $N\ge N_0$, then $\mathbb S_N$ admits no out-going flow.

To prove the claim by contradiction, we assume that for a sequence $N_k$ such that $\lim_{k\to \infty}N_k=\infty$, there is an out-going flow from $\mathbb S_{N_k}$ to a state $\vec x_k \in \mathbb S_{N_k}^c$. We find a sequence of reactions with which $X_{N_k}$ returns to $\mathbb S_{N_k}$ from $\vec x_k$ for all large enough $k$.

First of all, note $\{X_i \to Y \} \subset \widetilde \Re$ because each reaction $X_i \to \emptyset \in \Re$ is converted to $X_i \to Y$ in any optimized slack network by its definition of an optimized slack network (See Section \ref{subsec:properties of slack}). Hence for each $k$, by firing the reactions $\{X_i \to Y \}$, $X_{N_k}$ can reach to $\vec 0=(0,0,\dots,0)$ from $\vec x_k$. 

The next aim is to show that for any fixed $k$ large enough, there exist a sequence of reactions $\{\tilde \nu_1\to \tilde \nu'_1,\dots \tilde \nu_n\to \tilde \nu'_n\} \subset \widetilde \Re$ such that  we have i) $\lambda^{N_k}_{\tilde \nu_i\to \tilde \nu'_i}(\vec x_k+\sum_{j=1}^{i-1}(\tilde \nu'_j-\tilde \nu_j))>0$ for each $i$, and ii) $\vec x_k+\sum_{j=1}^{n}(q(\tilde \nu'_j)-q(\tilde \nu_j)) \in \mathbb S_{N_k}$. This means that $X_{N_k}$ can reach to $\mathbb S_{N_k}$ from $\vec x_k$ along the reactions $\tilde \nu_i\to \tilde \nu_i$ in that order. For this aim, we make a use of the irreducibility of $\mathbb S$ since $\mathbb S$ so that there exists a sequence of reactions $\{\nu_1\to \nu'_1,\cdots,\nu_n\to \nu'_n\} \subset \Re$ such that   i) $\lambda_{\nu_i\to \nu'_i}(\vec 0+\sum_{j=1}^{i-1}(\nu'_j-\nu_j))>0$ for each $i$ and ii) $\vec 0+\sum_{j=1}^{n}(\nu'_j-\nu_j) \in \mathbb S_{N_1}$. By making $N_0\ge N_1$ large enough, we have if $N_k\ge N_0$, then
\begin{align*}
   \vec w\cdot \left (\vec 0 + \sum_{j=1}^{i}(\nu'_j-\nu_j) \right ) < N-c \quad \text{for each $i$.}
\end{align*}
 Hence by Lemma \ref{lem2}, for each reaction $\tilde \nu_i \to \tilde \nu'_i$ such that $q(\tilde \nu_i)=\nu_i$ and $q(\tilde \nu'_i)=\nu'_i$ we have 
\begin{align}\label{eq:condition 1 for closedness}
    \lambda^{N}_{\tilde \nu_i\to \tilde \nu'_i}\left (\vec 0+\sum_{j=1}^{i-1}(q(\tilde \nu'_j)-q(\tilde \nu_j)) \right )>0 \quad \text{for each $i$}. 
\end{align}
 Finally since $q(\tilde \nu_i)=\nu_i$ and $q(\tilde \nu'_i)=\nu'_i$ for each $i$, we have
\begin{align}\label{eq:condition 2 for closedness}
   \vec 0+\sum_{j=1}^{i-1}(q(\tilde \nu'_j)-q(\tilde \nu_j)) \in \mathbb S_{N_0}\subseteq \mathbb S_{N}.
   \end{align}
Hence \eqref{eq:condition 1 for closedness} and \eqref{eq:condition 2 for closedness} imply that if $N_k \ge N_0$, then $X_{N_k}$ can re-enter to $\mathbb S_{N_k}$ from $\vec 0$ with positive probability.

Consequently, we constructed a sequence of reactions along which $X_{N_k}$ can re-enter $\mathbb S_{N_k}$ from $\vec x_k$ for any $k$ such that $N_k\ge N_0$. Since each $\mathbb S_{N_k}$ is a communication class, $\vec x_k\in \mathbb S_{N_k}$ and ,in turn, this contradictions to the assumption that $\vec x_k \in \mathbb S_{N_k}$. Hence the claim holds so that $\mathbb S_N$ is closed for any $N$ large enough. 
\hfill $\square$

\section{Proofs for Lemma \ref{lem:cb} and Theorem \ref{theorem:relation between two}}

\textbf{Proof of Lemma \ref{lem:cb}: }
In the deterministic model $\tilde{\vec x}(t)=(\tilde x_1(t),\dots,\tilde x_d(t),y_1(t),\dots,y_m(t))^\top$ associated with $(\widetilde \S,\widetilde \C,\widetilde \Re,\Lambda^N)$, if we fix $y_i(t)\equiv 1$ for all $i=1,2,\dots,m$, then $\tilde x(t)$ follows the same ODE system as the deterministic model $x(t)$ associated with the original network $(\S,\C,\Re,\Lambda)$ does. Therefore $\tilde c=(c^*,1,1,\dots,1)^\top$ is a complex balanced steady state for $\tilde x(t)$. It has been shown that the existence of a single positive complex balanced steady state implies that all other positive steady states are complex balanced \cite{feinberg1972complex}, therefore completing the proof.  

\hfill $\square$

\textbf{Proof of Theorem \ref{theorem:relation between two}: }
Let $\vec w_i$ and $N_i$ be the conservative vector and the conservative quantity of the slack network $(\widetilde \S,\widetilde \C,\widetilde \Re,\Lambda^N)$, respectively.
Then $\pi$ is a stationary solution of the chemical master equation \eqref{eq:master} of $X$,
\begin{align*}
    &\sum_{\nu\to\nu' \in \Re}\lambda_{\nu\to \nu'}(\vec x-\nu'+\nu)\pi(\vec x-\nu'+\nu)\\
    &=\sum_{\nu\to\nu' \in \Re}\lambda_{\nu\to \nu'}(\vec x)\pi(\vec x).
\end{align*}
Especially, as shown in Theorem 6.4, Anderson et al\cite{anderson2010product}, $\pi$ satisfies the \emph{stochastic complex balance} for $X$:  for each complex $\nu^* \in \Re$ and a state $\vec x$,
\begin{align}
\begin{split}\label{eq:sto cx bal}
    &\sum_{\substack{\nu\to \nu' \in \Re \\ \nu'=\nu^*}}\lambda_{\nu\to \nu'}(\vec x-\nu'+\nu)\pi(\vec x-\nu'+\nu)\\
    &=\sum_{\substack{\nu\to \nu' \in \Re \\ \nu=\nu^*}}\lambda_{\nu\to \nu'}(\vec x)\pi(\vec x).
    \end{split}
\end{align}

Let $q$ be the one-to-one projection function $q$ that we defined for the proof of Theorem \ref{theorem:all out flows}. Note that each reaction $\tilde \nu\to \tilde \nu'\in \widetilde \Re$ is defined as it satisfies the conservative law 
\begin{align}\label{eq:use conv law}
\vec w_i \cdot (\nu'-\nu)+\tilde \nu'_{d+i}-\nu_{d+i}=0,    
\end{align}
 where $q(\tilde \nu)=\nu$ and $q(\tilde \nu')=\nu'$. Then $\pi$ also satisfies the stochastic complex balance for $X_N$ because \eqref{eq:sto cx bal} and \eqref{eq:use conv law} imply that for each complex $\tilde \nu^*\in \widetilde \Re$ and a state $\vec x$,
\begin{align*}
        &\sum_{\substack{\tilde \nu\to \tilde \nu'; \in \widetilde \Re  \\ \tilde \nu'=\tilde \nu^*}}\lambda^N_{\tilde \nu\to \tilde \nu'}(\vec x-q(\nu')+q(\nu))\pi(\vec x-q(\nu')+q(\nu))\\
        &= \sum_{\substack{ \nu\to  \nu' \in \Re \\ \tilde \nu'=\tilde \nu^*}}\lambda_{\tilde \nu\to \tilde \nu'}(\vec x-\nu'+\nu)\pi(\vec x-\nu'+\nu)\prod_{i=1}^r\mathbbm{1}_{\{N_i-\vec w_i \cdot(\vec x-\nu'+\nu) \ge \tilde \nu_{d+i}\}}\\
        &= \sum_{\substack{ \nu\to  \nu' \in \Re \\ \tilde \nu'=\tilde \nu^*}}\lambda_{\tilde \nu\to \tilde \nu'}(\vec x-\nu'+\nu)\pi(\vec x-\nu'+\nu)\prod_{i=1}^r\mathbbm{1}_{\{N_i-\vec w_i \cdot \vec x \ge \tilde \nu^*_{d+i}\}}\\
        &=\sum_{\substack{\nu\to \nu' \in \Re \\ \nu=\nu^*}}\lambda_{\nu\to \nu'}(\vec x)\pi(\vec x)\mathbbm{1}_{\{N_i-\vec w_i \cdot \vec x \ge \tilde \nu^*_{d+i}\}}\\
         &=\sum_{\substack{\tilde \nu\to \tilde \nu' \in \widetilde \Re  \\ \tilde \nu=\tilde \nu^*}}\lambda^N_{\tilde \nu\to \tilde \nu'}(\vec x)\pi(\vec x).
\end{align*}
Then by summing up for each $\tilde \nu^*\in \widetilde \Re$, 
\begin{align*}
&\sum_{\tilde \nu\to \tilde \nu'; \in \widetilde \Re}\lambda^N_{\tilde \nu\to \tilde \nu'}(\vec x-q(\nu')+q(\nu))\pi(\vec x-q(\nu')+q(\nu))\\
&=\sum_{\tilde \nu\to \tilde \nu' \in \widetilde \Re }\lambda^N_{\tilde \nu\to \tilde \nu'}(\vec x)\pi(\vec x),
\end{align*}
and, in turn, $\pi_N$ is a stationary solution of the chemical master equation of $X_N$. 
Since the state space of $X$ and $X_N$ differ, $M_N$ is a constant such that the sum of $M_N\pi(\vec x)$ over the state space of $X_N$ is one.
Then $\pi_N=M_N\pi$. $\hfill \square$

\section{Lypunov Functions for Example Systems}\label{appenx:example1}

\subsection{Lotka-Volterra with Migration}
\label{sec:LyapunovLotkaVoltera}

We construct a Lyapunov function satisfying the conditions of Theorem \ref{thm:Lyapunov} for the Lotka-Volterra model with migration (Figure \ref{fig:GLV}A). First, for both the original model and the slack network, we assume $A(0)=a_0$, $B(0)=b_0$ that are not depend on $N$.

Let $V(\vec x)=e^{w\cdot \vec x}$ with $w=(1,1)^\top$. The condition in Theorem \ref{thm:Lyapunov} obviously holds. Note that the slack network of Figure \ref{fig:GLV}A admits a conservation law $A+B+Y=N$ and the state space of the system $\mathbb S_N$ is $\{(a,b) \ | \ a+b\le N\}$. Then condition 2 in Theorem \ref{thm:Lyapunov} also follows by the definition of $V$.

Now, to show condition 3 in Theorem \ref{thm:Lyapunov}, we first note that for $\vec x=(a,b)$,
\begin{align}
\begin{split}\label{eq:LV detail1}
    &\sum_{\nu \to \nu'\in \Re}\lambda_{\nu\to \nu'}(\vec x)(V(\vec x +\nu'-\nu)-V(\vec x))\\
    &=e^{w\cdot \vec x}\sum_{\nu \to \nu'\in \Re}\lambda_{\nu\to \nu'}(\vec x)(e^{w\cdot (\nu'-\nu)}-1)\\
    &=V(\vec x)\Big ( \kappa_1 b (e^{-1}-1)+\kappa_2(e-1)+\kappa_3(e^{-1}-1)\\ 
    & \ \ +\kappa_4 a (e^{-1}-1)+\kappa_5 a(e-1) \Big ).
    \end{split}
\end{align}
Since we assumed $\frac{\kappa_4}{\kappa_5}=3$ (see caption of Figure \ref{fig:GLV}), each coefficient of $a$ and $b$ is negative. Thus, there exists $M>0$  such that for each $\vec x \in \{ (a,b) \ | \ a> M \text{ or } b \ge M\}$, the left-hand side in \eqref{eq:LV detail1} satisfies
\begin{align*}
  &\sum_{\nu \to \nu'\in \Re}\lambda_{\nu\to \nu'}(\vec x)(V(\vec x +\nu'-\nu)-V(\vec x))\\
  &\le -CV(\vec x)
\end{align*}
for some $C>0$. Letting 
\[D=(C+1)\displaystyle\max _{a\le M \text{ and } b \le M}\sum_{\nu \to \nu'\in \Re}\lambda_{\nu\to \nu'}(\vec x)(V(\vec x +\nu'-\nu)-V(\vec x)),\] condition 3 in Theorem \ref{thm:Lyapunov} holds with $C$ and $D$. 

\subsection{Protein Synthesis with Slow Toggle Switch}
\label{sec:LyapunovToggle}

We also make a use of the Lyapunov approach shown in Theorem \ref{thm:Lyapunov} to prove that the first passage time of the slack system in Figure \ref{fig:Toggle}B converges to the original first passage time, as the truncation size $N$ goes infinity.
Let $X_N=(X,Z,D^X_0,D^X_1,D^Z_0,D^Z_1)$ be the stochastic system associated with the slack system. Recall that the slack network admits a conservation relation $w\cdot X\le N$ with $w=(1,1,0,0,0,0)$. Hence we define a Lyapunov function $V(\vec x)=e^{x+y}$ where $\vec x=(x,z,dx_0,dx_1,dz_0,dz_1)$. By the definition of $V$, it is obvious that conditions 1 and 2 in Theorem \ref{thm:Lyapunov} hold. So we show that condition 3 in Theorem holds.  

For a reaction $\nu\to\nu' \in \Re_I:=\{Z+D^X_0\to D^X_1, X\to 0, X+D^Z_0\to D^Z_1, Z\to 0\}$, it is clear that the term $V(\vec x+\nu'-\nu)-V(\vec x)$ is negative. For each reaction $\nu\to \nu'$ in $\Re^c_I$, it is also clear that the term $V(\vec x+\nu'-\nu)-V(\vec x)$ is positive. However, the reaction intensity for a reaction in $\Re_I$ is linear in either $x$ and $z$, while the reaction intensity for a reaction in $\Re^c_I$ is constant. Therefore there exists a constant $M>0$ such that if $x>M$ or $z>M$, then
\begin{align*}
  &\sum_{\nu \to \nu'\in \Re}\lambda_{\nu\to \nu'}(\vec x)(V(\vec x +\nu'-\nu)-V(\vec x))\\
  &\le -CV(\vec x)
\end{align*}
for some $C>0$. Hence, letting \[D=(C+1)\displaystyle\max _{x\le M \text{ and } z \le M}\sum_{\nu \to \nu'\in \Re}\lambda_{\nu\to \nu'}(\vec x)(V(\vec x +\nu'-\nu)-V(\vec x)),\] condition 3 in Theorem \ref{thm:Lyapunov} holds with $C$ and $D$. 


\bibliographystyle{plain}
\bibliography{slack}

\end{document}